\documentclass[10pt]{article}
\usepackage{lmodern}
\usepackage{amsmath}
\usepackage[T1]{fontenc}
\usepackage[utf8]{inputenc}
\usepackage{authblk}
\usepackage{amsfonts}
\usepackage{graphicx}
\usepackage{rotating}
\usepackage{amssymb}
\usepackage[english]{babel}
\usepackage{xcolor}
\usepackage{amsthm}
\usepackage{graphicx}
\usepackage{mathrsfs}
\usepackage{makecell}
\usepackage{microtype}
\usepackage{mathscinet}
\usepackage{array}
\usepackage{multirow}
\usepackage{enumerate}
\usepackage[cal=boondoxo,bb=ams]{mathalfa}
\usepackage{hyperref}
\hypersetup{hidelinks}
\usepackage{booktabs}

\newcounter{question}

\newenvironment{question}
{
	\refstepcounter{question}
	\par\medskip
	\noindent\textbf{Question \Roman{question}.}\itshape
}
{
	\par\medskip
}

\newtheorem{theorem}{Theorem}[section]
\newtheorem{prop}{Proposition}[section]
\newtheorem{lemma}{Lemma}[section]
\newtheorem{coro}{Corollary}[section]
\newtheorem{remark}{Remark}[section]

\newcommand{\ml}{\mathcal}
\newcommand{\mb}{\mathbb}
\DeclareMathOperator{\intt}{int}

\DeclareMathOperator{\bdd}{bdd}

\def\XXint#1#2#3{{\setbox0=\hbox{$#1{#2#3}{\int}$ }
		\vcenter{\hbox{$#2#3$ }}\kern-.6\wd0}}

\title{Large time intrinsic growth and asymptotic behavior for\\ the classical Timoshenko system}
\author[1]{Wenhui Chen\thanks{Wenhui Chen (wenhui.chen.math@gmail.com)}}
\affil[1]{School of Mathematics and Information Science, Guangzhou University,\authorcr 510006 Guangzhou, P. R. China}

\date{}

\begin{document}
		\maketitle

		\begin{abstract}
			In this paper,
			we investigate the large time behavior
			of solutions to the classical Timoshenko system
			in the whole space $\mathbb{R}$.
			Although the system is conservative
			and its natural energy is conserved in time,
			the transversal displacement $\varphi$
			and the rotation angle $\psi$
			exhibit intrinsic polynomial growths.
			We establish sharp $L^p-L^q$ estimates
			for the solutions
			and show that the growth mechanism
			originates from the interaction
			between quadratic oscillations
			and singular low-frequency amplitudes
			of different orders.
			Furthermore,
			we prove the optimality
			of the obtained growth rates
			under a nontrivial zeroth-moment condition
			on the initial data,
			while additional moment cancellations with a nontrivial first-moment condition
			lead to lower-order growth regimes.
			As a consequence,
			we derive large time asymptotic profiles
			related to an effective
			plate-type dispersive structure
			hidden in the low-frequency regime
			of the classical Timoshenko system. We also discuss the relation with the dissipative Timoshenko system through a large time vanishing dissipation limit for time-normalized solutions.
			\medskip
			\\
\noindent\textbf{Keywords:}
Timoshenko system,
intrinsic growth,
sharp asymptotics,
plate-type structure,
large time asymptotic profile, vanishing dissipation limit
			\\
			
			\noindent\textbf{AMS Classification (2020)}
			35L52, 	35B40, 35Q74
		\end{abstract}
\fontsize{12}{15}
\selectfont

\section{Introduction}\setcounter{equation}{0}\label{Section-Introduction}

\hspace{5mm}
In this paper,
we investigate the large time behavior
of solutions to the following classical Timoshenko system (or, the so-called conservative Timoshenko system)
in the whole space $\mathbb{R}$:
\begin{align}\label{Eq-Timoshenko}
	\begin{cases}
		\rho\,\varphi_{tt}-K(\varphi_x-\psi)_x=0,&x\in\mathbb{R},\ t>0,\\
		I_{\rho}\psi_{tt}-EI\psi_{xx}-K(\varphi_x-\psi)=0,&x\in\mathbb{R},\ t>0,\\
		(\varphi,\varphi_t)(0,x)=(\varphi_0,\varphi_1)(x),&x\in\mathbb{R},\\
		(\psi,\psi_t)(0,x)=(\psi_0,\psi_1)(x),&x\in\mathbb{R},
	\end{cases}
\end{align}
where the unknown functions
$\varphi=\varphi(t,x)$
and $\psi=\psi(t,x)$
represent the transversal displacement
and the rotation angle, respectively.
The positive constants
$\rho$, $I_\rho$, $E$, $I$, and $K$
denote the mass density,
the rotary inertia,
the Young modulus,
the moment of inertia of the cross section,
and the shear stiffness (cf. \cite{Graff=1975}), respectively.

The classical Timoshenko system
was established in the early 20th century by \cite{Timoshenko-01,Timoshenko-02},
and remains one of the cornerstones
of structural mechanics,
because it simultaneously captures
the transversal motion,
the shear deformation,
and the rotational inertia effects of beams.
To describe the propagation
and low-frequency structures of the system later, we introduce the following characteristic quantities.
\begin{table}[http]
	\centering
	\caption{Characteristic quantities for the Timoshenko system}
	\begin{tabular}{ll}
		\toprule
		Characteristic quantity & Notation\\
		\midrule
		Shear-wave speed
		& $c_{\mathrm{S}}:=\sqrt{\frac{K}{\rho}}$\\[1mm]
		
		Rotational-wave speed
		& $c_{\mathrm{R}}:=\sqrt{\frac{EI}{I_{\rho}}}$\\[1mm]
		
		Shear-rotation oscillation frequency
		& $c_{\mathrm{O}}:=\sqrt{\frac{K}{I_{\rho}}}$\\[1mm]
		
		Bending dispersion coefficient
		& $c_{\mathrm{D}}:=\sqrt{\frac{EI}{\rho}}$\\
		\bottomrule
	\end{tabular}
	\label{Table_2}
\end{table}

\noindent
The quantities
$c_{\mathrm{O}}$
and
$c_{\mathrm{D}}$
play essential roles
in the low-frequency structure
and the large time asymptotic behavior
of solutions to the classical Timoshenko system \eqref{Eq-Timoshenko}.
In particular,
they characterize
the oscillatory
and dispersive effects
appearing in the present paper.

Most existing studies on the Timoshenko system
have focused on dissipative mechanisms,
stability structures,
and decay properties of damped models.
In contrast, the large time asymptotic behavior
of solutions to the classical
conservative Timoshenko system
in the whole space $\mathbb{R}$
remains far less understood due to the lack of crucial damping mechanisms.

To understand the large time behavior
of the classical Timoshenko system
\eqref{Eq-Timoshenko},
it is important to compare it
with the corresponding dissipative Timoshenko system as follows:
\begin{align}\label{Eq-dissipative-Timoshenko}
	\begin{cases}
		\rho\,\varphi_{tt}-K(\varphi_x-\psi)_x=0,&x\in\mb{R},\ t>0,\\
		I_{\rho}\psi_{tt}-EI\psi_{xx}-K(\varphi_x-\psi)+\gamma\psi_t=0,&x\in\mb{R},\ t>0,\\
		(\varphi,\varphi_t)(0,x)=(\varphi_0,\varphi_1)(x),&x\in\mathbb{R},\\
		(\psi,\psi_t)(0,x)=(\psi_0,\psi_1)(x),&x\in\mathbb{R},
	\end{cases}
\end{align}
where $\gamma>0$
denotes the frictional damping coefficient. The dissipative Timoshenko system
has been extensively studied
in connection with
stability structures,
regularity-loss phenomena,
asymptotic behavior,
and decay properties.
In particular,
most existing asymptotic theories
for the Timoshenko system
have been developed
in the dissipative setting.
In the case of non-equal wave speeds
$c_{\mathrm{S}}\neq c_{\mathrm{R}}$,
the dissipative structure exhibits
the so-called regularity-loss phenomenon,
which was clarified in
\cite{Ide-Haramoto-Kawashima=2008}
by means of energy methods
in the Fourier space.
Subsequent studies established
refined decay estimates
and asymptotic profiles
for the dissipative Timoshenko system
and related models
(see, for example,
\cite{Ammar-Ben-Mun-Racke=2003,Ra-Fer-Santo-Castro=2005,Munoz-Racke=2008,Ide-Haramoto-Kawashima=2008,Fe-Racke=2009,Racke-Said=2013,Mori-Kawashima=2016,Guesmia-Messaoudi=2023}
and references therein). Recently,
\cite{Chen=2026}
discovered a new large time growth phenomenon
for the dissipative Timoshenko system \eqref{Eq-dissipative-Timoshenko}, to be specific,
\[
\|\varphi(t,\cdot)\|_{L^2}\approx t^{\frac34}
\ \ \mbox{and}\ \ 
\|\psi(t,\cdot)\|_{L^2}\approx t^{\frac14}
\ \ \mbox{for}\ \ 
t\gg1.
\]
Moreover,
the corresponding large time asymptotic profiles
were also identified, which are recalled in Proposition \ref{Prop-01} and Proposition \ref{Prop-02}.
These observations naturally lead
to the following question.
\begin{question}
	What is the mechanism responsible
	for the large time growth in the 
	dissipative Timoshenko system
	\eqref{Eq-dissipative-Timoshenko}?
\end{question}
\noindent Since the dissipative system
still exhibits polynomial growth
of the solutions themselves,
it becomes important to clarify
whether such growth originates
from the damping mechanism
or from the intrinsic structure
of the Timoshenko system.

The present paper shows that
the large time growth phenomenon
is already contained
in the conservative structure
of the classical Timoshenko system itself.
In particular,
the frictional damping mechanism
in the dissipative Timoshenko system \eqref{Eq-dissipative-Timoshenko}
is not the origin of the growth behavior.
Instead,
the growth is generated by the interaction between
quadratic oscillations
and singular low-frequency amplitudes
appearing in the Timoshenko system \eqref{Eq-Timoshenko}. Moreover, the coincidence between the growth structures
of the conservative and dissipative Timoshenko systems
naturally raises the following question.
\begin{question}
Whether the dissipative dynamics converge
to the conservative ones in the small damping regime as $\gamma\to0$?
\end{question}
\noindent Motivated by this observation,
we further discuss
a large time vanishing dissipation limit
for time-normalized solutions
in the final section.

It is well-known that
the classical Timoshenko system
\eqref{Eq-Timoshenko}
possesses the conserved natural energy
\begin{align*}
	E(t):=
	\frac{1}{2}\left(
	\rho\|\varphi_t(t,\cdot)\|_{L^2}^2
	+I_{\rho}\|\psi_t(t,\cdot)\|_{L^2}^2
	+EI\|\psi_x(t,\cdot)\|_{L^2}^2
	+K\|\varphi_x(t,\cdot)-\psi(t,\cdot)\|_{L^2}^2
	\right),
\end{align*}
satisfying
\[
E(t)\equiv E(0)
\ \ \mbox{for any}\ \ t>0.
\]
However,
the conservation of energy itself
does not provide
a detailed description
of the large time behavior
of the transversal displacement
$\varphi$
and the rotation angle
$\psi$.
This naturally leads
to the following question.
\begin{question}
Can one describe the precise large time behavior of $\varphi$ and $\psi$ themselves?
\end{question}
\noindent This problem goes beyond
the standard energy theory,
since the conserved energy
$E(t)$
does not control
the displacement variables
$\|\varphi(t,\cdot)\|_{L^2}$
and $\|\psi(t,\cdot)\|_{L^2}$ themselves in $\mb{R}$.
In particular,
their large time behavior
cannot be captured
by standard energy methods.

One of the purposes of the present paper
is to answer this question
by establishing sharp $L^p-L^q$
growth estimates
for solutions to \eqref{Eq-Timoshenko}.
More precisely,
for suitable initial data
in $H^{s,p}$
with
$1\leqslant p\leqslant2\leqslant q\leqslant+\infty$,
we derive
\begin{align*}
	\|\varphi(t,\cdot)\|_{L^q}
	&\lesssim
	(1+t)^{1-\frac{1}{2}\left(\frac{1}{p}-\frac{1}{q}\right)},
	\\
	\|\psi(t,\cdot)\|_{L^q}
	&\lesssim
	(1+t)^{\frac{1}{2}-\frac{1}{2}\left(\frac{1}{p}-\frac{1}{q}\right)}.
\end{align*}
Furthermore,
in the case
$(p,q)=(1,2)$,
we establish
the optimal large time growth rates
\begin{align*}
	\|\varphi(t,\cdot)\|_{L^2}
	\approx
	t^{\frac{3}{4}}
	\ \ \mbox{and}\ \
	\|\psi(t,\cdot)\|_{L^2}
	\approx
	t^{\frac{1}{4}}
\end{align*}
for
$t\gg1$,
under the nontrivial moment condition
$
\int_{\mathbb{R}}
\varphi_1(x)\,\mathrm{d}x
\neq0$.
Moreover, the additional moment cancellation $\int_{\mb{R}}\varphi_1(x)\,\mathrm{d}x=0$ leads to lower-order growth regimes.
Our results reveal that the large time growth is an intrinsic phenomenon generated by the interaction between quadratic oscillations and singular low-frequency amplitudes appearing in the classical Timoshenko system. The above questions, therefore, are answered through the sharp
$L^p-L^q$ growth estimates obtained in Section \ref{Sec-Lq},
the optimal asymptotic profiles derived in Section \ref{Section-L2-est},
and the large time vanishing dissipation limit established
in Section \ref{Sec-Final}.

The derivation of the sharp
$L^q$-growth estimates
is highly nontrivial.
In the low-frequency region,
the analysis requires
a delicate treatment
of the oscillatory
and dispersive structures
of the Timoshenko system.
In the high-frequency region,
the coupled propagation structure
must be combined
with refined dyadic decomposition arguments
to identify the dominant contributions.
The proof of the optimal
$L^2$-growth estimates
is even more delicate.
In contrast to the dissipative case
studied in \cite{Chen=2026},
the classical Timoshenko system
does not contain
the additional low-frequency regularization
generated by the fractional damping structure.
Consequently,
the Fourier analysis
developed in \cite{Chen=2026}
is no longer directly applicable.
To overcome this difficulty,
we adapt several ideas
from the wave and plate equations
\cite{Ikehata=2023,Chen-Takeda=2023,Ikehata=2024,Takeda=2026}
to the present strongly coupled system,
while avoiding explicit calculations
of the characteristic roots.
This requires
a refined analysis
of the coupled oscillatory structure
generated by the Timoshenko system.

\medskip
\noindent\textbf{\large Notation.}
Let the generic positive constants
$c$ and $C$,
independent of $t$,
vary from line to line.
The notation
$f\lesssim g$
means that there exists a constant
$C>0$
such that
$f\leqslant Cg$.
Similarly,
$f\gtrsim g$
means
$g\lesssim f$.
We write
$f\approx g$
if both
$f\lesssim g$
and
$g\lesssim f$
hold.
We denote by
$f\ast_{(x)}g$
the convolution of
$f$ and $g$
with respect to the spatial variable $x$. We denote by
$\widehat{f}=\mathcal{F}_{x\to\xi}(f)$
the Fourier transform of $f$
in the spatial variable $x$,
and by
$\mathcal{F}^{-1}_{\xi\to x}$
its inverse.
Moreover,
we introduce the following zones
in the Fourier space:
\begin{align*}
	\mathcal{Z}_{\mathrm{int}}(\varepsilon_0)
	&:=
	\{\xi\in\mathbb{R}:\ |\xi|\leqslant\varepsilon_0\},
	\\
	\mathcal{Z}_{\mathrm{bdd}}(\varepsilon_0,N_0)
	&:=
	\{\xi\in\mathbb{R}:\ \varepsilon_0\leqslant|\xi|\leqslant N_0\},
	\\
	\mathcal{Z}_{\mathrm{ext}}(N_0)
	&:=
	\{\xi\in\mathbb{R}:\ |\xi|\geqslant N_0\},
\end{align*}
where
$\varepsilon_0>0$
is sufficiently small
and
$N_0>0$
is sufficiently large.
Let
$\chi_{\mathrm{int}}(\xi)$,
$\chi_{\mathrm{bdd}}(\xi)$,
$\chi_{\mathrm{ext}}(\xi)$
be smooth cut-off functions
supported in
$\mathcal{Z}_{\mathrm{int}}(\varepsilon_0)$,
$\mathcal{Z}_{\mathrm{bdd}}(\frac{\varepsilon_0}{2},2N_0)$,
$\mathcal{Z}_{\mathrm{ext}}(N_0)$,
respectively,
satisfying
\[
\chi_{\mathrm{bdd}}(\xi)
=
1-\chi_{\mathrm{int}}(\xi)-\chi_{\mathrm{ext}}(\xi)
\ \ 
\mbox{for all}
\ \ 
\xi\in\mathbb{R}.
\]
For simplicity,
we write
$\chi:=\chi_{\mathrm{int}}$
and define
\begin{align*}
	\|f\|_{L^q_{\chi}}
	:=
	\big\|\chi(D)f\big\|_{L^q}
	\ \ \mbox{and}\ \ 
	\|f\|_{L^q_{1-\chi}}
	:=
	\big\|
	\big(1-\chi(D)\big)f
	\big\|_{L^q}
\end{align*}
to denote the corresponding localized norms.
The Bessel potential space
is denoted by
\[
H^{s,p}
:=
\left\{
f\in\mathcal{S}'
:\ 
(1-\partial_x^2)^{\frac{s}{2}}f\in L^p
\right\},
\]
where
$s\in\mathbb{R}$
and
$1\leqslant p\leqslant+\infty$.
Finally,
we define the weighted
$L^1$ space by
\begin{align*}
	L^{1,\sigma}
	:=
	\left\{
	f\in L^1:
	\ \ 
	\|f\|_{L^{1,\sigma}}
	:=
	\int_{\mb{R}}
	(1+|x|)^{\sigma}|f(x)|\,\mathrm{d}x
	<+\infty
	\right\},
\end{align*} where $\sigma\in\mb{N}_0$. Note that $L^{1,0}\equiv L^1$.
The zeroth moment for $f\in L^1$,
and first moment for $f\in L^{1,1}$, respectively, are defined by
\begin{align*}
	P_f
	:=
	\int_{\mathbb{R}}
	f(x)\,\mathrm{d}x
	\ \ \mbox{and}\ \ 
	M_f
	:=
	\int_{\mathbb{R}}
	(-x)f(x)\,\mathrm{d}x.
\end{align*}

\section{Main results}\setcounter{equation}{0}\label{Section-Main-Results}

\subsection{$L^p-L^q$ growth estimates}

\hspace{5mm}
For brevity,
we introduce two data spaces
\begin{align*}
	X_{p,q}^{\epsilon}
	&:=
	H^{s_{p,q}+\epsilon,p}
	\times
	H^{s_{p,q}-1+\epsilon,p}
	\ \ \mbox{and}\ \ 
	Y_{p,q}^{\epsilon}
	:=
	H^{s_{p,q}-1+\epsilon,p}
	\times
	H^{s_{p,q}-2+\epsilon,p}
\end{align*}
with the index
\begin{align*}
	s_{p,q}
	:=
	\frac{3}{2}
	\left(
	\frac{1}{p}-\frac{1}{q}
	\right)
	\ \ 
	\mbox{for}
	\ \ 
	1\leqslant p\leqslant2\leqslant q\leqslant+\infty.
\end{align*}
Our first result reveals
the intrinsic polynomial growth
of the displacement variables
in the $L^p-L^q$ framework.

\begin{theorem}\label{Theorem-Lq}
	Suppose that the initial data
	belong to the corresponding spaces
	appearing on the right-hand side
	of the estimates below,
	with an arbitrarily small constant
	$\epsilon>0$.
	Then,
	 the transversal displacement
	$\varphi$
	and
	the rotation angle
$\psi$
	to the classical Timoshenko system
	\eqref{Eq-Timoshenko}
	satisfy the following
	$L^p-L^q$ growth estimates:
	\begin{itemize}
		\item
		in the case of non-equal speed
		$c_{\mathrm{S}}\neq c_{\mathrm{R}}$,
		\begin{align*}
			\|\varphi(t,\cdot)\|_{L^q}
			&\lesssim
			(1+t)^{
				1-\frac12\left(\frac1p-\frac1q\right)
			}
			\|(\varphi_0,\varphi_1)\|_{X_{p,q}^{\epsilon}}
			+
			(1+t)^{
				\frac12-\frac12\left(\frac1p-\frac1q\right)
			}
			\|(\psi_0,\psi_1)\|_{Y_{p,q}^{\epsilon}},
			\\
			\|\psi(t,\cdot)\|_{L^q}
			&\lesssim
			(1+t)^{
				\frac12-\frac12\left(\frac1p-\frac1q\right)
			}
			\|(\varphi_0,\varphi_1)\|_{Y_{p,q}^{\epsilon}}
	+
			(1+t)^{
				-\frac12\left(\frac1p-\frac1q\right)
			}
			\|(\psi_0,\psi_1)\|_{X_{p,q}^{\epsilon}};
		\end{align*}
		
		\item
		in the case of equal speed
		$c_{\mathrm{S}}=c_{\mathrm{R}}$,
		\begin{align*}
			\|\varphi(t,\cdot)\|_{L^q}
			&\lesssim
			(1+t)^{
				1-\frac12\left(\frac1p-\frac1q\right)
			}
			\|(\varphi_0,\varphi_1)\|_{X_{p,q}^{\epsilon}}
			+
			(1+t)^{
				\frac12-\frac12\left(\frac1p-\frac1q\right)
			}
			\|(\psi_0,\psi_1)\|_{X_{p,q}^{\epsilon}},
			\\
			\|\psi(t,\cdot)\|_{L^q}
			&\lesssim
			(1+t)^{
				\frac12-\frac12\left(\frac1p-\frac1q\right)
			}
			\|(\varphi_0,\varphi_1)\|_{X_{p,q}^{\epsilon}}
			+
			(1+t)^{
				-\frac12\left(\frac1p-\frac1q\right)
			}
			\|(\psi_0,\psi_1)\|_{X_{p,q}^{\epsilon}}.
		\end{align*}
	\end{itemize}
\end{theorem}
\begin{remark}
	The arbitrarily small loss
	$\epsilon>0$
	is technical
	and arises from the dyadic decomposition
	in the high-frequency analysis.
	More precisely,
	it is caused by the embedding
	$
	H^{s+\epsilon,p}\hookrightarrow B^s_{p,1}
	$.
\end{remark}

\begin{remark}
	The difference between
	the equal speed case
	and the non-equal speed case
	in Theorem \ref{Theorem-Lq}
	has a natural interpretation
	from the viewpoint
	of propagation mechanisms.
	When
	$c_{\mathrm{S}}\neq c_{\mathrm{R}}$,
	the shear and rotational waves propagate
	with genuinely different speeds,
	which produces a stronger separation
	of high-frequency modes
and
stronger high-frequency interactions,
which require additional regularity
in the corresponding estimates.
	In contrast,
	when
	$c_{\mathrm{S}}=c_{\mathrm{R}}$,
	the two principal propagation mechanisms
	are coherent at the leading order.
	This propagation coherence weakens
	the high-frequency loss
	and explains
	the improved regularity requirements
	in the equal speed configuration.
\end{remark}

\begin{remark}
	The derived $L^p-L^q$ estimates
	can be interpreted
	through the interaction between
	quadratic oscillations
	and singular low-frequency amplitudes.
	More precisely,
	the dispersive scaling produces
	the decay factor
	\[
	(1+t)^{
		-\frac12\left(\frac1p-\frac1q\right)},
	\]
	whereas the amplitude singularities
	generate the polynomial growth factors.
	Therefore,
	the polynomial growth
	is an intrinsic phenomenon
	of the conservative Timoshenko system.
	In particular,
	it should not be interpreted
	as an instability phenomenon,
	since the characteristic roots
	remain purely imaginary
	and the natural energy
	is conserved in time.
\end{remark}

\subsection{Sharp large time asymptotics}

\hspace{5mm}
For later convenience,
for
$\sigma\in\{0,1\}$,
we introduce
\begin{align*}
	Z_1^{\sigma}
	&:=
	L^2\times(H^{-1}\cap L^{1,\sigma}),
	\\
	Z_2^{\sigma}
	&:=
	H^{-1}\times(H^{-2}\cap L^{1,\sigma}),
	\\
	Z_3
	&:=
	L^2\times H^{-1}.
\end{align*}
Our second contribution
reveals the sharp large time asymptotic behavior
of solutions to
the classical Timoshenko system.
In particular,
the leading asymptotic profiles
are governed by
the zeroth moment
of the initial velocity
$\varphi_1$,
while additional moment cancellations
lead to lower-order growth regimes.
These asymptotic characterizations
confirm the sharpness
of the corresponding growth estimates
in the special case
$(p,q)=(1,2)$ of Theorem 2.1.

\begin{theorem}\label{Thm-01}
	Let
	$\sigma\in\{0,1\}$.
	Suppose that
	$(\varphi_0,\varphi_1)\in Z_1^{\sigma}$
	and
	$(\psi_0,\psi_1)\in Z_{\psi,\star}^{\sigma}$ for the classical Timoshenko system
	\eqref{Eq-Timoshenko},
	where
	\[
	Z_{\psi,\star}^{\sigma}
	:=
	\begin{cases}
		Z_2^{\sigma}
		&\mbox{if}\ \ c_{\mathrm S}\neq c_{\mathrm R},
		\\
		Z_1^{\sigma}
		&\mbox{if}\ \ c_{\mathrm S}=c_{\mathrm R}.
	\end{cases}
	\]
	Then, the following statements hold
	for sufficiently large time.
	\begin{itemize}
		\item
		If
		$P_{\varphi_1}\neq0$,
		then the transversal displacement
		$\varphi$
		satisfies the following optimal growth estimate:
		\begin{align*}
			t^{\frac34}|P_{\varphi_1}|
			\lesssim
			\|\varphi(t,\cdot)\|_{L^2}
			\lesssim
			t^{\frac34}
			\|(\varphi_0,\varphi_1)\|_{Z_1^0}
			+
			t^{\frac14}
			\|(\psi_0,\psi_1)\|_{Z_{\psi,\star}^0},
		\end{align*}
		and the following asymptotic relation:
		\begin{align*}
			\lim\limits_{t\to+\infty}t^{-\frac34}\left\|
			\varphi(t,\cdot)
			-
			\sqrt{t}\,
			\ml{G}_{0}
			\left(
			\tfrac{\cdot}{\sqrt t}
			\right)
			P_{\varphi_1}
			\right\|_{L^2}
			=0,
		\end{align*}
		where the intrinsic plate-type asymptotic profile is given by
		\begin{align*}
			\ml{G}_{0}(y)
			:=
			\ml{F}^{-1}_{\eta\to y}
			\left(
			\frac{
				\sin(c_{\mathrm D}|\eta|^2)
			}{
				c_{\mathrm D}|\eta|^2
			}
			\right).
		\end{align*}
		
		\item
		If
		$P_{\varphi_1}=0$
		but
		$M_{\varphi_1}-\frac{I_{\rho}}{\rho}P_{\psi_1}\neq0$,
		then the transversal displacement
		$\varphi$
		satisfies the following optimal growth estimate:
		\begin{align*}
			t^{\frac14}
			\left|
			M_{\varphi_1}
			-
			\frac{I_{\rho}}{\rho}P_{\psi_1}
			\right|
			\lesssim
			\|\varphi(t,\cdot)\|_{L^2}
			\lesssim
			t^{\frac14}
			\|(\varphi_0,\varphi_1)\|_{Z_1^1}
			+
			t^{\frac14}
			\|(\psi_0,\psi_1)\|_{Z_{\psi,\star}^1},
		\end{align*}
		and the following asymptotic relation:
		\begin{align*}
			\lim\limits_{t\to+\infty}t^{-\frac14}\left\|
			\varphi(t,\cdot)
			-
			\ml{G}_{1}
			\left(
			\tfrac{\cdot}{\sqrt{t}}
			\right)
			\left(
			M_{\varphi_1}
			-
			\frac{I_{\rho}}{\rho}P_{\psi_1}
			\right)
			\right\|_{L^2}
			=0,
		\end{align*}
		where the derivative of intrinsic plate-type asymptotic profile is given by
		\begin{align*}
			\ml{G}_{1}(y)
			:=\partial_y\ml{G}_0(y)=
			\ml{F}^{-1}_{\eta\to y}
			\left(i\eta
			\frac{
				\sin(c_{\mathrm D}|\eta|^2)
			}{
				c_{\mathrm D}|\eta|^2
			}
			\right).
		\end{align*}
	\end{itemize}
\end{theorem}
\begin{remark}
	When
	$P_{\varphi_1}=0$
	and
	$M_{\varphi_1}-\frac{I_{\rho}}{\rho}P_{\psi_1}=0$,
	by additionally assuming $\varphi_1\in L^{1,2}$ and $\psi_1\in L^{1,1}$, we are able to prove the boundedness $\|\varphi(t,\cdot)\|_{L^2}\lesssim 1$, which lies beyond the scope of the present paper.
\end{remark}

\begin{coro}\label{Coro-L-infty}
	Suppose that
	$(\varphi_0,\varphi_1)\in Z_1^{0}$
	and
	$(\psi_0,\psi_1)\in Z_{\psi,\star}^{0}$
	such that
	$P_{\varphi_1}\neq0$
	for the classical Timoshenko system
	\eqref{Eq-Timoshenko}.
	Assume additionally
	\begin{align*}
		\mathrm{supp}\,
		(\varphi_0,\varphi_1,\psi_0,\psi_1)
		\subset[-L,L]
		\ \ \mbox{for some}\ \ 
		L>0.
	\end{align*}
	Then,
	the transversal displacement
	$\varphi$
	satisfies the following $L^q$-growth estimate with any $q\geqslant 2$:
	\begin{align*}
		\|\varphi(t,\cdot)\|_{L^q}
		\gtrsim
		t^{\frac14+\frac{1}{q}}|P_{\varphi_1}|
	\end{align*}
	for sufficiently large time.
\end{coro}

\begin{theorem}\label{Thm-02}
	Suppose that
	$(\varphi_0,\varphi_1)\in Z_{\varphi,\sharp}^0$
	and
	$(\psi_0,\psi_1)\in Z_{3}$
	such that
	$P_{\varphi_1}\neq0$
	for the classical Timoshenko system
	\eqref{Eq-Timoshenko},
	where
	\[
	Z^0_{\varphi,\sharp}
	:=
	\begin{cases}
		Z_2^0
		&\mbox{if}\ \ c_{\mathrm S}\neq c_{\mathrm R},
		\\
		Z_1^0
		&\mbox{if}\ \ c_{\mathrm S}=c_{\mathrm R}.
	\end{cases}
	\]
	Then,
	the rotation angle
	$\psi$
	satisfies the following optimal growth estimate:
	\begin{align*}
		t^{\frac14}|P_{\varphi_1}|
		\lesssim
		\|\psi(t,\cdot)\|_{L^2}
		\lesssim
		t^{\frac14}
		\|(\varphi_0,\varphi_1)\|_{Z^0_{\varphi,\sharp}}
		+
		\|(\psi_0,\psi_1)\|_{Z_{\psi,3}}
	\end{align*}
	for sufficiently large time,
	and the following asymptotic relation:
	\begin{align*}
		\lim\limits_{t\to+\infty}t^{-\frac14}\left\|
		\psi(t,\cdot)
		-
		\ml{G}_{1}
		\left(
		\tfrac{\cdot}{\sqrt t}
		\right)
		P_{\varphi_1}
		\right\|_{L^2}
		=0.
	\end{align*}
\end{theorem}

\begin{remark}
	When
	$P_{\varphi_1}=0$,
	by additionally assuming $\varphi_1\in L^{1,1}$, we are able to prove the boundedness $\|\psi(t,\cdot)\|_{L^2}\lesssim 1$, which lies beyond the scope of the present paper.
\end{remark}
\begin{remark}
	The polynomial growth
	derived in this paper
	is generated entirely
	by the low-frequency structure
	of the classical Timoshenko system.
	In particular,
	the low-frequency zone
	$\mathcal{Z}_{\mathrm{int}}(\varepsilon_0)$
	produces the leading asymptotic behavior.
\end{remark}
\begin{remark}
	The asymptotic profiles
	$\ml{G}_0(y)$
	and
	$\ml{G}_1(y)$
	are generated by
	the oscillatory multiplier
	\[
	\chi_{\intt}(\xi)\frac{\sin(c_{\mathrm D} |\xi|^2 t)}
	{c_{\mathrm D} |\xi|^2},
	\]
	which reveals
	a hidden plate-type structure
	in the low-frequency regime
	of the classical Timoshenko system \eqref{Eq-Timoshenko}.
	Therefore,
	the leading large time behavior
	is governed not by the hyperbolic wave propagation itself,
	but by the associated quadratic oscillations
	and singular low-frequency amplitudes.
\end{remark}

\begin{remark}\label{Rem-2.8}
	The polynomial growth derived in this paper
	should not be interpreted
	as an instability phenomenon
	of the classical Timoshenko system.
	Indeed,
	all characteristic roots
	remain purely imaginary,
	and the natural energy
	is exactly conserved in time.
	The coincidence of the growth rates
	and asymptotic profiles
	with those for the dissipative Timoshenko system \eqref{Eq-dissipative-Timoshenko}
	in \cite{Chen=2026}
	shows that the frictional damping mechanism
	does not generate
	the leading large time growth.
	Instead,
	the polynomial growth
	is already contained intrinsically
	in the conservative structure
	of the classical Timoshenko system. This relation is further clarified
	in the final section,
	where we discuss
	a large time vanishing dissipation limit
	between the dissipative
	and classical Timoshenko systems.
\end{remark}

\section{Preliminaries}\setcounter{equation}{0}\label{Section-Preliminary}

\subsection{Representation of solutions in the Fourier space}

\hspace{5mm}
To reveal the hidden plate-type structure
of the classical Timoshenko system
\eqref{Eq-Timoshenko},
we first reduce the coupled system
to a fourth-order equation.
More precisely,
following the reduction procedure
in \cite[Section 3.1]{Chen=2026}
(see also
\cite{Chen-Ikehata=2023,Chen-Takeda=2023}
for related thermoelastic systems),
we eliminate the coupling terms
by applying the Klein-Gordon operator
$(I_{\rho}\,\partial_t^2-EI\partial_x^2+K)$
to \eqref{Eq-Timoshenko}$_1$
and the free wave operator
$(\rho\,\partial_t^2-K\partial_x^2)$
to \eqref{Eq-Timoshenko}$_2$.

As a consequence,
each component
$u\in\{\varphi,\psi\}$
satisfies the following
fourth-order equation:
\begin{align*}
	\begin{cases}
		\displaystyle{
			\frac{I_{\rho}}{K}u_{tttt}
			-\left(
			\frac{EI}{K}
			+\frac{I_{\rho}}{\rho}
			\right)
			u_{ttxx}
			+\frac{EI}{\rho}u_{xxxx}
			+u_{tt}
			=0,
		}
		&x\in\mathbb{R},\ t>0,
		\\[0.5em]
		(u,u_t,u_{tt},u_{ttt})(0,x)=(u_0,u_1,u_2,u_3)(x),
		&x\in\mathbb{R}.
	\end{cases}
\end{align*}
where the higher-order initial data
are given by
\begin{align*}
	u_2(x)
	:=
	\begin{cases}
		\displaystyle{
			\frac{K}{\rho}
			\big(
			\varphi''_0(x)-\psi'_0(x)
			\big)
		}
		&\mbox{if}\quad u=\varphi,
		\\[1em]
		\displaystyle{
			\frac{1}{I_{\rho}}
			\Big[
			EI\psi''_0(x)
			+
			K\big(
			\varphi'_0(x)-\psi_0(x)
			\big)
			\Big]
		}
		&\mbox{if}\quad u=\psi,
	\end{cases}
\end{align*}
and
\begin{align*}
	u_3(x)
	:=
	\begin{cases}
		\displaystyle{
			\frac{K}{\rho}
			\big(
			\varphi''_1(x)-\psi'_1(x)
			\big)
		}
		&\mbox{if}\quad u=\varphi,
		\\[1em]
		\displaystyle{
			\frac{1}{I_{\rho}}
			\Big[
			EI\psi''_1(x)
			+
			K\big(
			\varphi'_1(x)-\psi_1(x)
			\big)
			\Big]
		}
		&\mbox{if}\quad u=\psi.
	\end{cases}
\end{align*}
\begin{remark}
	The partial differential operator
	\begin{align*}
		\mathcal{L}_{\mathrm{Cla}}(\partial_t,\partial_x)
		:=
		\frac{I_{\rho}}{K}\partial_t^4
		-
		\left(
		\frac{EI}{K}
		+
		\frac{I_{\rho}}{\rho}
		\right)
		\partial_t^2\partial_x^2
		+
		\frac{EI}{\rho}\partial_x^4
		+
		\partial_t^2
	\end{align*}
	associated with the classical Timoshenko system
	\eqref{Eq-Timoshenko}
	has the same leading structure
	as the operator corresponding to
	the dissipative Timoshenko system
	\eqref{Eq-dissipative-Timoshenko}
	(see \cite[Equation (3.1)]{Chen=2026}),
	namely,
	\begin{align*}
		\mathcal{L}_{\mathrm{Dis}}
		(\partial_t,\partial_x;\gamma)
		:=
		\mathcal{L}_{\mathrm{Cla}}(\partial_t,\partial_x)
		+
		\frac{\gamma}{K}
		\partial_t
		\left(
		\partial_t^2
		-
		\frac{K}{\rho}\partial_x^2
		\right).
	\end{align*}
	This common structure explains
	why both systems exhibit
	the same leading asymptotic growth rates.
	However,
	when $\gamma=0$,
	the dissipative correction disappears,
	and the low-frequency regularization mechanism
	is no longer available.
\end{remark}
\begin{remark}
	Although both
	$\varphi$
	and
	$\psi$
	satisfy the same fourth-order equation,
	their asymptotic behaviors differ significantly.
	In particular,
	the growth rates,
	the asymptotic profiles,
	and the required regularities
	of the initial data
	are different.
	This phenomenon originates from
	the distinct higher-order initial data
	$u_2$
	and
	$u_3$
	generated by the coupling structure
	of the original Timoshenko system.
	Consequently,
	the low-frequency singular structures
	appearing in the Fourier representations
	of
	$\varphi$
	and
	$\psi$
	are also different.
\end{remark}

Applying the partial Fourier transform
with respect to $x$,
one finds that
$\widehat{u}$
solves
\begin{align*}
	\begin{cases}
		\displaystyle{
			\frac{I_{\rho}}{K}\widehat{u}_{tttt}
			+\left[
			\left(
			\frac{EI}{K}
			+
			\frac{I_{\rho}}{\rho}
			\right)|\xi|^2
			+1
			\right]\widehat{u}_{tt}
			+
			\frac{EI}{\rho}|\xi|^4\widehat{u}
			=0,
		}
		&\xi\in\mathbb{R},\ t>0,
		\\[0.5em]
		(\widehat{u},
		\widehat{u}_t,
		\widehat{u}_{tt},
		\widehat{u}_{ttt})(0,\xi)
		=
		(\widehat{u}_0,
		\widehat{u}_1,
		\widehat{u}_2,
		\widehat{u}_3)(\xi),
		&\xi\in\mathbb{R}.
	\end{cases}
\end{align*}
Its characteristic equation is
\begin{align}\label{Quartic-Eq}
	\frac{I_{\rho}}{K}\lambda^4
	+
	\left[
	\left(
	\frac{EI}{K}
	+
	\frac{I_{\rho}}{\rho}
	\right)|\xi|^2
	+1
	\right]\lambda^2
	+
	\frac{EI}{\rho}|\xi|^4
	=
	0.
\end{align}
Since the discriminant associated with
$\lambda^2$
is strictly positive,
the quartic equation
\eqref{Quartic-Eq}
has four purely imaginary roots.
We denote them by
\[
\lambda_{1,2}
=
\pm i\lambda_{\mathrm{I},1}
\ \ \mbox{and}\ \ 
\lambda_{3,4}
=
\pm i\lambda_{\mathrm{I},2},
\]
where
\begin{align*}
	\lambda_{\mathrm{I},1}^2
	&=
	\frac{K}{2I_{\rho}}
	\left(
	A(|\xi|)
	+
	\sqrt{
		[A(|\xi|)]^2
		-
		\frac{4EI}{\rho K}|\xi|^4
	}\
	\right),
	\\[0.5em]
	\lambda_{\mathrm{I},2}^2
	&=
	\frac{K}{2I_{\rho}}
	\left(
	A(|\xi|)
	-
	\sqrt{
		[A(|\xi|)]^2
		-
		\frac{4EI}{\rho K}|\xi|^4
	}\
	\right),
\end{align*}
with
\[
A(|\xi|)
:=
\left(
\frac{EI}{K}
+
\frac{I_{\rho}}{\rho}
\right)|\xi|^2
+1.
\]
The low-frequency asymptotic behavior
of the Timoshenko system
is generated by
$\lambda_{\mathrm{I},2}$,
whose quadratic oscillatory structure
produces the hidden plate-type behavior
appearing throughout the paper.

Let us derive the explicit Fourier representations
of the two components
$\widehat{\varphi}$
and
$\widehat{\psi}$.
Using the expressions of
$\widehat{u}_2(\xi)$
and
$\widehat{u}_3(\xi)$,
and applying Cramer's rule,
we obtain, for
$\xi\neq0$,
\begin{align}
	\widehat{\varphi}&= 
	\frac{\frac{K}{\rho}(|\xi|^{2}\widehat{\varphi}_0 + i\xi\widehat{\psi}_0) - \lambda_{\mathrm{I},2}^{2}\,\widehat{\varphi}_0}
	{\lambda_{\mathrm{I},1}^{2} - \lambda_{\mathrm{I},2}^{2}}\cos(\lambda_{\mathrm{I},1}t)+\frac{\frac{K}{\rho}(|\xi|^{2}\widehat{\varphi}_1 + i\xi\widehat{\psi}_1) - \lambda_{\mathrm{I},2}^{2}\,\widehat{\varphi}_1}
	{\lambda_{\mathrm{I},1}^{2} - \lambda_{\mathrm{I},2}^{2}}\,\frac{\sin(\lambda_{\mathrm{I},1}t)}{\lambda_{\mathrm{I},1}} \notag\\
	&\,\quad- \frac{\frac{K}{\rho}(|\xi|^{2}\widehat{\varphi}_0 + i\xi\widehat{\psi}_0) - \lambda_{\mathrm{I},1}^{2}\,\widehat{\varphi}_0}
	{\lambda_{\mathrm{I},1}^{2}-\lambda_{\mathrm{I},2}^{2} }\cos(\lambda_{\mathrm{I},2}t)-\frac{\frac{K}{\rho}(|\xi|^{2}\widehat{\varphi}_1 + i\xi\widehat{\psi}_1) - \lambda_{\mathrm{I},1}^{2}\,\widehat{\varphi}_1}
	{ \lambda_{\mathrm{I},1}^{2}-\lambda_{\mathrm{I},2}^{2}}\,\frac{\sin(\lambda_{\mathrm{I},2}t)}{\lambda_{\mathrm{I},2}}\notag\\
	&=:\sum\limits_{j\in\{0,1\}}\widehat{K}_{\varphi,j}(t,|\xi|)\widehat{\varphi}_j+\sum\limits_{j\in\{0,1\}}\widehat{K}_{\psi,j}(t,|\xi|)\widehat{\psi}_j,\label{Rep-varphi}
\end{align}
moreover,
\begin{align}
	\widehat{\psi}&= 
	-\frac{ \frac{K}{I_{\rho}}i\xi\widehat{\varphi}_0 +\big( \lambda_{\mathrm{I},2}^{2} - \frac{E I|\xi|^{2} + K}{I_{\rho}} \big)\widehat{\psi}_0 }
	{  \lambda_{\mathrm{I},1}^{2}-\lambda_{\mathrm{I},2}^{2} } \cos(\lambda_{\mathrm{I},1}t)- \frac{ \frac{K}{I_{\rho}}i\xi\widehat{\varphi}_1 + \big( \lambda_{\mathrm{I},2}^{2} - \frac{E I|\xi|^{2} + K}{I_{\rho}} \big)\widehat{\psi}_1 }
	{ \lambda_{\mathrm{I},1}^{2}-\lambda_{\mathrm{I},2}^{2} }\, \frac{\sin(\lambda_{\mathrm{I},1}t)}{\lambda_{\mathrm{I},1}}\notag\\
	&\,\quad+ \frac{ \frac{K}{I_{\rho}}i\xi\widehat{\varphi}_0 + \big( \lambda_{\mathrm{I},1}^{2} - \frac{E I|\xi|^{2} + K}{I_{\rho}} \big)\widehat{\psi}_0 }
	{ \lambda_{\mathrm{I},1}^{2} - \lambda_{\mathrm{I},2}^{2} } \cos(\lambda_{\mathrm{I},2}t)+ \frac{ \frac{K}{I_{\rho}}i\xi\widehat{\varphi}_1 + \big( \lambda_{\mathrm{I},1}^{2} - \frac{E I|\xi|^{2} + K}{I_{\rho}} \big)\widehat{\psi}_1 }
	{ \lambda_{\mathrm{I},1}^{2} - \lambda_{\mathrm{I},2}^{2} }\, \frac{\sin(\lambda_{\mathrm{I},2}t)}{\lambda_{\mathrm{I},2}}\notag\\
	&=:\sum\limits_{j\in\{0,1\}}\widehat{G}_{\varphi,j}(t,|\xi|)\widehat{\varphi}_j+\sum\limits_{j\in\{0,1\}}\widehat{G}_{\psi,j}(t,|\xi|)\widehat{\psi}_j.\label{Rep-psi}
\end{align}
Here and in what follows,
the quotient
$
\frac{\sin(\lambda_{\mathrm{I},2}t)}
{\lambda_{\mathrm{I},2}}
$
at
$\xi=0$
is understood in the limiting sense.
The Fourier multipliers introduced in
\eqref{Rep-varphi}
and
\eqref{Rep-psi}
will be analyzed separately
in the low-frequency and high-frequency regions.

\subsection{Asymptotic expansions of characteristic roots}

\hspace{5mm}
The characteristic roots exhibit different asymptotic structures in the low-frequency and high-frequency zones.

\begin{description}
	
	\item[Low-frequencies: $|\xi|\leqslant\varepsilon_0\ll1$.]
	
	A direct expansion yields
	\begin{align*}
		\lambda_{\mathrm{I},1}
		&=
		c_{\mathrm O}
		+
		\frac{1}{2c_{\mathrm O}}
		\left(
		\frac{EI}{I_\rho}
		+
		\frac{K}{\rho}
		\right)|\xi|^2
		+
		O(|\xi|^4),
		\\[0.5em]
		\lambda_{\mathrm{I},2}
		&=
		c_{\mathrm D}|\xi|^2
		+
		O(|\xi|^6).
	\end{align*}
	
	In particular,
	the quadratic oscillatory structure
	generated by
	$\lambda_{\mathrm{I},2}$
	produces the hidden plate-type behavior
	appearing in the large time asymptotics.
	
	Moreover,
	\begin{align*}
		\lambda_{\mathrm{I},1}^{2}
		-
		\frac{EI|\xi|^2+K}{I_\rho}
		=
		\frac{K}{\rho}|\xi|^2
		+
		O(|\xi|^4),
	\end{align*}
	where the cancellation
	of the zeroth-order terms
	plays an important role
	in reducing the singular structure
	of the Fourier amplitudes
	associated with
	$\widehat{\psi}$.
	
	\item[High-frequencies: $|\xi|\geqslant N_0\gg1$.]
	
	If
	$c_{\mathrm S}\neq c_{\mathrm R}$,
	then
	\begin{align*}
		\lambda_{\mathrm{I},1}
		&=
		\max\{c_{\mathrm S},c_{\mathrm R}\}|\xi|
		+
		O(|\xi|^{-1}),
		\\[0.5em]
		\lambda_{\mathrm{I},2}
		&=
		\min\{c_{\mathrm S},c_{\mathrm R}\}|\xi|
		+
		O(|\xi|^{-1}).
	\end{align*}
	
	If
	$c_{\mathrm S}=c_{\mathrm R}$,
	then
	\begin{align*}
		\lambda_{\mathrm{I},1}
		=
		c_{\mathrm S}|\xi|
		+
		\frac{c_{\mathrm O}}{2}
		+
		O(|\xi|^{-1}),
		\\[0.3em]
		\lambda_{\mathrm{I},2}
		=
		c_{\mathrm S}|\xi|
		-
		\frac{c_{\mathrm O}}{2}
		+
		O(|\xi|^{-1}).
	\end{align*}
	
Furthermore, one has the asymptotic relation
\begin{align}\label{Eq-01}
	\lambda_{\mathrm{I},1}^2-\lambda_{\mathrm{I},2}^2\approx
	\begin{cases}
		|\xi|^2
		&\mbox{if}\quad c_{\mathrm{S}}\neq c_{\mathrm{R}},\\[0.3em]
		|\xi|
		&\mbox{if}\quad c_{\mathrm{S}}=c_{\mathrm{R}}.
	\end{cases}
\end{align}
due to an additional cancellation
of the leading high-frequency terms.
This distinction reflects
the propagation coherence
in the equal speed configuration
and explains
the improved regularity requirements
in the equal speed case.
\end{description}

\subsection{Analytical tools}
\label{Sub-section-main-tools}

\hspace{5mm}
As a basic tool
for estimating oscillatory integrals soon afterwards,
we recall the classical van der Corput lemma
(cf.~\cite{Stein=1993}).

\begin{lemma}
	\label{Lemma-Van-der-Corput}
	
	Let
	$I\subset\mathbb{R}$
	be an interval.
	Suppose that
	$\Psi\in \ml{C}^2(I)$
	is real-valued and satisfies
	\[
	|\Psi''(\xi)|
	\geqslant
	\mu
	>0
	\ \ 
	\mbox{for all}
	\ \ 
	\xi\in I,
	\]
	and let
	$b\in \ml{C}^1(I)$
	be complex-valued.
	Then,
	\begin{align*}
		\left|
		\int_I
		\mathrm{e}^{i\Psi(\xi)}\,
		b(\xi)
		\,\mathrm d\xi
		\right|
		\lesssim
		\mu^{-1/2}
		\left(
		\|b\|_{L^\infty(I)}
		+
		\|b'\|_{L^1(I)}
		\right),
	\end{align*}
	where the implicit constant
	is independent of
	$\mu$.
\end{lemma}

We also recall
the following estimate
for weighted $L^1$ data
(cf.~\cite{Ikehata=2014}).

\begin{lemma}\label{Lemma-Ikehata}
	Suppose that
	$f\in L^{1,1}$.
	Then,
	\begin{align*}
		|\widehat{f}(\xi)|
		\lesssim
		|P_f|
		+
		|\xi|\,\|f\|_{L^{1,1}}.
	\end{align*}
\end{lemma}

\section{$L^q$-growth estimates of solutions}\setcounter{equation}{0}\label{Sec-Lq}

\subsection{Low-frequency estimates of dominant Fourier multipliers}

\hspace{5mm}
The representations
\eqref{Rep-varphi}
and
\eqref{Rep-psi}
show that
the leading low-frequency contributions
are generated by
\begin{align*}
\mbox{the oscillatory multiplier} \ \ \frac{\sin(\lambda_{\mathrm I,2}t)}
{\lambda_{\mathrm I,2}},
\end{align*}
whose low-frequency behavior
is governed by (see Lemma \ref{Lemma-Approx} later)
\begin{align*}
\mbox{the quadratic oscillation}\ \ \frac{\sin(c_{\mathrm D}|\xi|^2t)}
{c_{\mathrm D}|\xi|^2}.
\end{align*}
The estimate below reveals
the precise balance
between the dispersive decay
generated by the oscillatory phase
and the polynomial growth
caused by the amplitude singularity.

\begin{lemma}
	\label{Lemma-Small}
	
	Suppose that
	$f\in L^p$
	and
	$\ell\in\{0,1,2\}$.
	Then,
	the Fourier multiplier satisfies
	the following $L^p-L^q$ estimate:
	\begin{align*}
		\left\|
		\mathcal F^{-1}_{\xi\to x}
		\left(
		|\xi|^\ell\,
		\frac{\sin(\lambda_{\mathrm I,2}t)}
		{\lambda_{\mathrm I,2}}
		\right)
		\ast_{(x)}
		f(\cdot)
		\right\|_{L^q_{\chi}}
		\lesssim
		(1+t)^{
			1-\frac{\ell}{2}
			-\frac12\left(\frac1p-\frac1q\right)
		}
		\|f\|_{L^p}
	\end{align*}
	for
	$1\leqslant p\leqslant2\leqslant q\leqslant+\infty$.
\end{lemma}
\begin{proof}
	Let us rewrite the Fourier multiplier as
	\begin{align*}
		\mathcal{K}_{1,\ell}(t,x)
		:=
		\chi_{\intt}(D)
		\mathcal{F}^{-1}_{\xi\to x}
		\left(
		|\xi|^{\ell}\,
		\frac{\sin(\lambda_{\mathrm{I},2}t)}
		{\lambda_{\mathrm{I},2}}
		\right)
		& =
		\int_{\mathbb{R}}
		\chi_{\intt}(\xi)\,
		\mathrm{e}^{ix\xi}\,
		|\xi|^{\ell}\,
		\frac{\sin(\lambda_{\mathrm{I},2}t)}
		{\lambda_{\mathrm{I},2}}
		\,\mathrm{d}\xi
		\\
		&=
		\int_{\mathbb{R}}
		\chi_{\intt}(\xi)\,
		\mathrm{e}^{ix\xi}\,
		|\xi|^{\ell}\,
		\frac{
			\sin\big(|\xi|^2g_0(\xi)t\big)
		}
		{
			|\xi|^2g_0(\xi)
		}
		\,\mathrm{d}\xi,
	\end{align*}
	where we denoted $\lambda_{\mathrm{I},2}
	=
	|\xi|^2g_0(\xi)$ with $g_0(0)=c_{\mathrm{D}}>0$.
	
	We first prove the limit case $(p,q)=(1,+\infty)$.
	Motivated by the quadratic oscillatory structure appearing in the low-frequency zone, we introduce the scaling variable $\eta=\sqrt{t}\,\xi$,
	which leads to
	\begin{align*}
		\mathcal{K}_{1,\ell}(t,x)
		=
		t^{\frac12-\frac{\ell}{2}}
		\widetilde{\mathcal{K}}_{1,\ell,t}
		\left(
		\tfrac{x}{\sqrt{t}}
		\right)
	\end{align*}
	with
	\begin{align*}
		\widetilde{\mathcal{K}}_{1,\ell,t}(y):=
		\int_{\mathbb{R}}\chi_{\intt}
		\left(\tfrac{\eta}{\sqrt{t}}\right)
		\mathrm{e}^{iy\eta}\,
		|\eta|^{\ell}\,
		\frac{
			\sin\big(
			|\eta|^2g_0(\frac{\eta}{\sqrt{t}})
			\big)
		}
		{
			|\eta|^2g_0(\frac{\eta}{\sqrt{t}})
		}
		\,\mathrm{d}\eta.
	\end{align*}
	We now derive estimates for
	$\widetilde{\mathcal{K}}_{1,\ell,t}(y)$
	uniformly in $y$.
	We separately estimate
	the bounded region
	$|\eta|\leqslant2$
	and the oscillatory region
	$|\eta|\geqslant1$.
	For the region
	$|\eta|\leqslant2$,
	the boundedness of
	$\frac{\sin y}{y}$
	and
	$\mathrm{e}^{iy\eta}$
	yields
	\begin{align*}
		\left|
		\int_{|\eta|\leqslant2}
		\chi_{\intt}
		\left(
		\tfrac{\eta}{\sqrt{t}}
		\right)
		\mathrm{e}^{iy\eta}\,
		|\eta|^{\ell}\,
		\frac{
			\sin\big(
			|\eta|^2g_0(\frac{\eta}{\sqrt{t}})
			\big)
		}
		{
			|\eta|^2g_0(\frac{\eta}{\sqrt{t}})
		}
		\,\mathrm{d}\eta
		\right|
		\lesssim
		\int_{|\eta|\leqslant2}
		|\eta|^{\ell}
		\,\mathrm{d}\eta
		\lesssim1.
	\end{align*}
	For the oscillatory region
	$|\eta|\geqslant1$,
	we rewrite the integral via Euler's formula as
	\begin{align*}
		&
		\int_{|\eta|\geqslant1}
		\chi_{\intt}
		\left(
		\tfrac{\eta}{\sqrt{t}}
		\right)
		\mathrm{e}^{iy\eta}\,
		|\eta|^{\ell}\,
		\frac{
			\sin\big(
			|\eta|^2g_0(\frac{\eta}{\sqrt{t}})
			\big)
		}
		{
			|\eta|^2g_0(\frac{\eta}{\sqrt{t}})
		}
		\,\mathrm{d}\eta
		=
		\frac{1}{2i}
		\int_{|\eta|\geqslant1}
		\left(
		\mathrm{e}^{i\Phi_+(y,\eta)}
		-
		\mathrm{e}^{i\Phi_-(y,\eta)}
		\right)
		b_{\ell,t}(\eta)
		\,\mathrm{d}\eta,
	\end{align*}
	where we considered
	\begin{align*}
		\Phi_{\pm}(y,\eta)
		:=
		y\eta
		\pm
		|\eta|^2g_0
		\left(
		\tfrac{\eta}{\sqrt{t}}
		\right)\ \ \mbox{and}\ \ 
		b_{\ell,t}(\eta)
		:=
		\chi_{\intt}
		\left(
		\tfrac{\eta}{\sqrt{t}}
		\right)
		\frac{
			|\eta|^{\ell-2}
		}
		{
			g_0(\frac{\eta}{\sqrt{t}})
		}.
	\end{align*}
	A direct computation gives
	\begin{align*}
		\left|
		\partial_{\eta}^2
		\Phi_{\pm}(y,\eta)
		\right|
		&=
		\left|
		2g_0
		\left(
		\tfrac{\eta}{\sqrt{t}}
		\right)
		+
		\frac{4\eta}{\sqrt{t}}
		g_0'
		\left(
		\tfrac{\eta}{\sqrt{t}}
		\right)
		+
		\frac{|\eta|^2}{t}
		g_0''
		\left(
		\tfrac{\eta}{\sqrt{t}}
		\right)
		\right|
		\\
		&\geqslant
		2
		\left|
		g_0
		\left(
		\tfrac{\eta}{\sqrt{t}}
		\right)
		\right|
		-
		\frac{|\eta|}{\sqrt{t}}
		\left|
		4g_0'
		\left(
		\tfrac{\eta}{\sqrt{t}}
		\right)
		+
		\frac{|\eta|}{\sqrt{t}}
		g_0''
		\left(
		\tfrac{\eta}{\sqrt{t}}
		\right)
		\right|
		\\
		&\geqslant
		C_0,
	\end{align*}
	where we used
	the continuity of $g_0$,
	the identity
	$g_0(0)=c_{\mathrm{D}}>0$,
	which guarantees
	a non-degenerate quadratic oscillation
	in the low-frequency regime,
	and the small support property of
	$\chi_{\intt}$.
	Since
	$|\eta|\geqslant1$,
	we immediately have
	\begin{align*}
		\|b_{\ell,t}\|_{L^\infty(|\eta|\geqslant1)}
		\lesssim1
	\end{align*}
	for
	$\ell\in\{0,1,2\}$.
	Moreover, the first-order derivative is controlled by
	\begin{align*}
		|b'_{\ell,t}(\eta)|
		&=
		\left|
		\chi_{\intt}'
		\left(
		\tfrac{\eta}{\sqrt{t}}
		\right)
		\frac{
			|\eta|^{\ell-2}
		}
		{
			\sqrt{t}\,
			g_0(\frac{\eta}{\sqrt{t}})
		}
		+
		\chi_{\intt}
		\left(
		\tfrac{\eta}{\sqrt{t}}
		\right)
		\frac{
			(\ell-2)|\eta|^{\ell-3}
		}
		{
			g_0(\frac{\eta}{\sqrt{t}})
		}
		-
		\chi_{\intt}
		\left(
		\tfrac{\eta}{\sqrt{t}}
		\right)
		\frac{
			|\eta|^{\ell-2}\,
			g_0'(\frac{\eta}{\sqrt{t}})
		}
		{
			\sqrt{t}\,
			[g_0(\frac{\eta}{\sqrt{t}})]^2
		}
		\right|
		\\
		&\lesssim
		|\eta|^{\ell-3}
		+
		\frac{1}{\sqrt{t}}
		|\eta|^{\ell-2}.
	\end{align*}
	Therefore, it leads to
	\begin{align*}
		\|b'_{\ell,t}\|_{L^1(|\eta|\geqslant1)}
		&\lesssim
		\int_{1\leqslant|\eta|\leqslant\varepsilon_0\sqrt{t}}
		\left(
		|\eta|^{\ell-3}
		+
		\frac{1}{\sqrt{t}}
		|\eta|^{\ell-2}
		\right)
		\mathrm{d}\eta
		\lesssim1
	\end{align*}
	for
	$\ell\in\{0,1,2\}$
	and
	$t\geqslant1$.
	By Lemma \ref{Lemma-Van-der-Corput},
	we conclude that
	\begin{align*}
		\left|
		\int_{|\eta|\geqslant1}
		\chi_{\intt}
		\left(
		\tfrac{\eta}{\sqrt{t}}
		\right)
		\mathrm{e}^{iy\eta}\,
		|\eta|^{\ell}\,
		\frac{
			\sin\big(
			|\eta|^2g_0(\frac{\eta}{\sqrt{t}})
			\big)
		}
		{
			|\eta|^2g_0(\frac{\eta}{\sqrt{t}})
		}
		\,\mathrm{d}\eta
		\right|
		&
		\lesssim
		\sum_{\pm}
		\left|
		\int_{|\eta|\geqslant1}
		\mathrm{e}^{i\Phi_{\pm}(y,\eta)}\,
		b_{\ell,t}(\eta)
		\,\mathrm{d}\eta
		\right|
		\\
		&
		\lesssim
		\|b_{\ell,t}\|_{L^\infty(|\eta|\geqslant1)}
		+
		\|b'_{\ell,t}\|_{L^1(|\eta|\geqslant1)}
		\\
		&
		\lesssim1.
	\end{align*}
	Hence,
	\begin{align*}
		\|\widetilde{\mathcal{K}}_{1,\ell,t}\|_{L^\infty}
		\lesssim1
	\end{align*}
	for all
	$\ell\in\{0,1,2\}$
	and
	$t\geqslant1$,
	namely,
	\begin{align*}
		\|\mathcal{K}_{1,\ell}(t,\cdot)\|_{L^\infty}
		\lesssim
		t^{\frac12-\frac{\ell}{2}}.
	\end{align*}
	The case
	$0<t\leqslant1$
	follows directly from
	\begin{align*}
		\|\mathcal{K}_{1,\ell}(t,\cdot)\|_{L^\infty}
		\lesssim
		t
		\int_{\mathbb{R}}
		\chi_{\intt}(\xi)
		|\xi|^\ell
		\,\mathrm{d}\xi
		\lesssim1.
	\end{align*}
	Consequently, a summary of them addresses
	\begin{align*}
		\|\mathcal{K}_{1,\ell}(t,\cdot)\|_{L^\infty}
		\lesssim
		(1+t)^{\frac12-\frac{\ell}{2}}
	\end{align*}
	for all
	$\ell\in\{0,1,2\}$
	and
	$t>0$.
	By Young's inequality,
	\begin{align}\label{Est-01}
		\|\mathcal{K}_{1,\ell}(t,\cdot)
		\ast_{(x)}f(\cdot)\|_{L^\infty}
		\lesssim
		(1+t)^{\frac12-\frac{\ell}{2}}
		\|f\|_{L^1}.
	\end{align}
	
	We next consider the case
	$(p,q)=(2,2)$.
	Separating the regions
	$|\xi|\leqslant t^{-\frac{1}{2}}$
	and
	$t^{-\frac{1}{2}}\leqslant|\xi|\leqslant\varepsilon_0$,
	we are able to derive
	\begin{align*}
		\|\widehat{\mathcal{K}}_{1,\ell}(t,\xi)\|_{L^\infty}
		&\lesssim
		\left\|
		|\xi|^\ell\,
		\frac{\sin(\lambda_{\mathrm{I},2}t)}
		{\lambda_{\mathrm{I},2}}
		\right\|_{L^\infty}
		\lesssim
		(1+t)^{1-\frac{\ell}{2}},
	\end{align*}
	where we used
	$|\sin y|\leqslant|y|$
	in the first region and
	$|\sin y|\leqslant1$
	in the second one.
	By Plancherel's identity,
	\begin{align}\label{Est-06}
		\|\mathcal{K}_{1,\ell}(t,\cdot)
		\ast_{(x)}f(\cdot)\|_{L^2}
		&\lesssim
		\|\widehat{\mathcal{K}}_{1,\ell}(t,\xi)\|_{L^\infty}
		\|f\|_{L^2}
		\notag\\
		&\lesssim
		(1+t)^{1-\frac{\ell}{2}}
		\|f\|_{L^2}.
	\end{align}
	
	The estimates
	\eqref{Est-01}
	and
	\eqref{Est-06}
	correspond respectively
	to the dispersive endpoint
	$(p,q)=(1,+\infty)$
	and the energy endpoint
	$(p,q)=(2,2)$. Finally,
	the Riesz-Thorin interpolation theorem
	between
	\eqref{Est-01}
	and
	\eqref{Est-06}
	completes the proof.
\end{proof}
\begin{remark}
	The growth factor
	$(1+t)^{1-\frac{\ell}{2}}$
	originates from the amplitude singularity
	$|\xi|^{\ell-2}$,
	whereas the dispersive scaling
$(1+t)^{-\frac12\left(\frac1p-\frac1q\right)}$ is generated by the quadratic oscillation.
\end{remark}

By applying the same argument as in
Lemma \ref{Lemma-Small},
one obtains the following auxiliary result.
Compared with
Lemma \ref{Lemma-Small},
the proof is simpler since the multiplier
$\cos(\lambda_{\mathrm{I},2}t)$
does not contain the singular factor
$\lambda_{\mathrm{I},2}^{-1}$
at
$|\xi|=0$.

\begin{lemma}\label{Lemma-Small-2}
	Suppose that
	$f\in L^p$
	and
	$\ell\in\{0,1,2\}$.
	Then,
	the Fourier multiplier satisfies
	the following $L^p-L^q$ estimate:
	\begin{align*}
		\left\|
		\mathcal{F}^{-1}_{\xi\to x}
		\left(
		|\xi|^{\ell}\cos(\lambda_{\mathrm{I},2}t)
		\right)
		\ast_{(x)}f(\cdot)
		\right\|_{L^q_{\chi}}
		\lesssim
		(1+t)^{
			-\frac{\ell}{2}
			-\frac{1}{2}\left(\frac{1}{p}-\frac{1}{q}\right)
		}
		\|f\|_{L^p}
	\end{align*}
	for
	$1\leqslant p\leqslant2\leqslant q\leqslant+\infty$.
\end{lemma}

\subsection{$L^p-L^q$ estimates of dominant Fourier multipliers for high-frequencies}

\hspace{5mm}
In contrast to the low-frequency regime,
the high-frequency multipliers
do not contain the strong singular structure
appearing in the low-frequency region,
since
$\lambda_{\mathrm I,2}\approx |\xi|$
for
$|\xi|\gg1$.
Consequently,
the polynomial growth disappears
in the high-frequency estimates.
We first derive
$L^p-L^q$ estimates
for the dominant oscillatory multipliers
associated with the sine terms.

\begin{lemma}\label{Lemma-Large}
	
	Suppose that
	$f\in H^{s,p}$
	with $s>
	\ell-1
	+
	\frac32\left(\frac1p-\frac1q\right)$ for $
	\ell\in\{0,1,2\}$.
	Then,
	the Fourier multiplier satisfies
	the following $L^p-L^q$ estimate:
	\begin{align*}
		\left\|
		\mathcal{F}^{-1}_{\xi\to x}
		\left(
		|\xi|^\ell\,
		\frac{\sin(\lambda_{\mathrm I,2}t)}
		{\lambda_{\mathrm I,2}}
		\right)
		\ast_{(x)}
		f(\cdot)
		\right\|_{L^q_{1-\chi}}
		\lesssim
		(1+t)^{
			-\frac12\left(\frac1p-\frac1q\right)
		}
		\|f\|_{H^{s,p}}
	\end{align*}
	for
	$1\leqslant p\leqslant2\leqslant q\leqslant+\infty$.
	The same estimate remains valid
	with
	$\lambda_{\mathrm I,2}$
	replaced by
	$\lambda_{\mathrm I,1}$.
	
\end{lemma}
\begin{proof}
	Let us rewrite the Fourier multiplier as
	\begin{align*}
		\mathcal{K}_{2,\ell}(t,x)
		&:=
		\big(1-\chi_{\intt}(D)\big)
		\mathcal{F}^{-1}_{\xi\to x}
		\left(
		|\xi|^{\ell}\,
		\frac{\sin(\lambda_{\mathrm{I},2}t)}
		{\lambda_{\mathrm{I},2}}
		\right)
		\\
		&\,=
		\frac{1}{2i}
		\sum_{j\geqslant j_0}
		\int_{\mathbb{R}}
		\chi_j(\xi)
		\,\mathrm{e}^{ix\xi}
		\left(
		\mathrm{e}^{i\lambda_{\mathrm{I},2}t}
		-
		\mathrm{e}^{-i\lambda_{\mathrm{I},2}t}
		\right)
		|\xi|^\ell\,
		\lambda_{\mathrm{I},2}^{-1}
		\,\mathrm d\xi,
	\end{align*}
	where we use the dyadic decomposition
	\[
	1-\chi_{\intt}(\xi)
	=
	\sum_{j\geqslant j_0}\chi_j(\xi)
	\ \ \mbox{with}\ \ 
	\chi_j(\xi):=\chi_0(2^{-j}\xi),
	\]
	and
	$|\xi|\approx 2^j$
	on
	$\operatorname{supp}\chi_j$.
	For each dyadic block,
	we consider
	\[
	\mathcal{K}_{2,\ell,j}^{\pm}(t,x)
	:=
	\int_{\mathbb{R}}
	\mathrm{e}^{i(x\xi\pm\lambda_{\mathrm{I},2}t)}\,
	a_j(\xi)
	\,\mathrm d\xi,
	\]
	where
	\[
	a_j(\xi)
	:=
	\chi_j(\xi)
	|\xi|^\ell\,
	\lambda_{\mathrm{I},2}^{-1}.
	\]
	For high-frequencies,
	the asymptotic expansion of
	$\lambda_{\mathrm{I},2}$
	yields $|\lambda_{\mathrm{I},2}''(\xi)|
	\approx
	|\xi|^{-3}
	\approx
	2^{-3j}$ on $\operatorname{supp}\chi_j$.
	Therefore,
	applying Lemma \ref{Lemma-Van-der-Corput}
	with
	$\mu\approx t\,2^{-3j}$,
	we obtain
	\begin{align*}
		\|\mathcal{K}_{2,\ell,j}^{\pm}(t,\cdot)\|_{L^\infty}
		&\lesssim
		t^{-\frac12}\,
		2^{\frac32j}
		\left(
		\|a_j\|_{L^\infty}
		+
		\|a_j'\|_{L^1}
		\right).
	\end{align*}
	Since
	$\lambda_{\mathrm{I},2}\approx |\xi|$
	and
	$|\xi|\approx 2^j$
	on
	$\operatorname{supp}\chi_j$,
	we have
	\[
	\|a_j\|_{L^\infty}
	+
	\|a_j'\|_{L^1}
	\lesssim
	2^{(\ell-1)j}.
	\]
	That is to say,
	\begin{align*}
		\|\mathcal{K}_{2,\ell,j}^{\pm}(t,\cdot)\|_{L^\infty}
		\lesssim
		t^{-\frac12}\,
		2^{(\ell+\frac12)j}.
	\end{align*}
	The dispersive decay
	$t^{-\frac{1}{2}}$
	originates from the oscillatory structure
	of the high-frequency phase,
	while the dyadic factor
	$2^{(\ell+\frac12)j}$
	reflects the regularity requirement
	generated by the multiplier amplitude.
	By Young's inequality
	and summing over
	$j$,
	we get
	\begin{align*}
		\|\mathcal{K}_{2,\ell}(t,\cdot)\ast_{(x)}f(\cdot)\|_{L^\infty}
		&\lesssim
		t^{-\frac12}
		\sum_{j\geqslant j_0}
		2^{(\ell+\frac12)j}
		\|\chi_j(D)f\|_{L^1}
		\\
		&\lesssim
		t^{-\frac12}
		\|f\|_{B^{\ell+\frac12}_{1,1}}.
	\end{align*}
	Using the Sobolev--Besov embedding 
$
	H^{s,1}
	\hookrightarrow
	B^{\ell+\frac12}_{1,1}
$
	for
	$s>\ell+\frac12$,
	we arrive at
	\begin{align}\label{Est-Large-01}
		\|\mathcal{K}_{2,\ell}(t,\cdot)\ast_{(x)}f(\cdot)\|_{L^\infty}
		\lesssim
		t^{-\frac12}
		\|f\|_{H^{s,1}}
	\end{align}
	for
	$s>\ell+\frac12$.
	
	On the other hand,
	by Plancherel's identity,
	\begin{align}
		\|\mathcal{K}_{2,\ell}(t,\cdot)\ast_{(x)}f(\cdot)\|_{L^2}
		&\lesssim
		\left\|\big(
	1-\chi_{\intt}(\xi)\big)
		|\xi|^\ell\,
		\lambda_{\mathrm{I},2}^{-1}
		\widehat f(\xi)
		\right\|_{L^2}
		\notag\\
		&\lesssim
		\|f\|_{H^{\ell-1,2}},\label{Est-Large-02}
	\end{align}
	because
	$\lambda_{\mathrm{I},2}\approx|\xi|$
	in the high-frequency region.
	
	The estimates
	\eqref{Est-Large-01}
	and
	\eqref{Est-Large-02}
	correspond respectively
	to the dispersive endpoint
	$(p,q)=(1,+\infty)$
	and the energy endpoint
	$(p,q)=(2,2)$.
		Interpolating between
	\eqref{Est-Large-01}
	and
	\eqref{Est-Large-02}
	by the Riesz--Thorin theorem
	yields
	\begin{align*}
		\|\mathcal{K}_{2,\ell}(t,\cdot)\ast_{(x)}f(\cdot)\|_{L^q}
		\lesssim
		t^{-\frac12\left(\frac1p-\frac1q\right)}
		\|f\|_{H^{s,p}},
	\end{align*}
	provided that
	$
	s>
	\ell-1
	+
	\frac32\left(\frac1p-\frac1q\right)
	$.
	Replacing
	$t^{-\frac12\left(\frac1p-\frac1q\right)}$
	by
	$(1+t)^{-\frac12\left(\frac1p-\frac1q\right)}$
	also covers the case
	$0<t\leqslant1$.
	This completes the proof.
\end{proof}

By the same argument as above,
one obtains the following auxiliary result
for the cosine multipliers.

\begin{lemma}\label{Lemma-Large-2}
	Suppose that
	$f\in H^{s,p}$
	with $s>
	\ell
	+
	\frac{3}{2}
	\left(
	\frac{1}{p}-\frac{1}{q}
	\right)$ for $\ell\in\{0,1,2\}$.
	Then,
	the Fourier multiplier satisfies
	the following $L^p-L^q$ estimate:
	\begin{align*}
		\left\|
		\mathcal{F}^{-1}_{\xi\to x}
		\left(
		|\xi|^{\ell}\cos(\lambda_{\mathrm{I},2}t)
		\right)
		\ast_{(x)}f(\cdot)
		\right\|_{L^q_{1-\chi}}
		\lesssim
		(1+t)^{
			-\frac{1}{2}\left(\frac{1}{p}-\frac{1}{q}\right)
		}
		\|f\|_{H^{s,p}}
	\end{align*}
	for
	$1\leqslant p\leqslant2\leqslant q\leqslant+\infty$.
	The same estimate remains valid
	with
	$\lambda_{\mathrm{I},2}$
	replaced by
	$\lambda_{\mathrm{I},1}$.
\end{lemma}

\subsection{Proof of Theorem \ref{Theorem-Lq}}

\hspace{5mm}
According to the representations of solutions
in the Fourier space,
we first estimate the low-frequency part.
The dominant contributions
are controlled by
Lemma \ref{Lemma-Small},
whereas the remaining terms
are controlled by
Lemma \ref{Lemma-Small-2}.
Consequently,
for any
$1\leqslant p\leqslant2\leqslant q\leqslant+\infty$,
we obtain
\begin{align*}
	\|\varphi(t,\cdot)\|_{L^q_{\chi}}
	&\lesssim
	(1+t)^{
		1-\frac12\left(\frac1p-\frac1q\right)
	}
	\|(\varphi_0,\varphi_1)\|_{L^p\times L^p}
	+
	(1+t)^{
		\frac12-\frac12\left(\frac1p-\frac1q\right)
	}
	\|(\psi_0,\psi_1)\|_{L^p\times L^p},
	\\
	\|\psi(t,\cdot)\|_{L^q_{\chi}}
	&\lesssim
	(1+t)^{
		\frac12-\frac12\left(\frac1p-\frac1q\right)
	}
	\|(\varphi_0,\varphi_1)\|_{L^p\times L^p}
	+
	(1+t)^{
		-\frac12\left(\frac1p-\frac1q\right)
	}
	\|(\psi_0,\psi_1)\|_{L^p\times L^p}.
\end{align*}

We next estimate
the part away from
the low-frequency zone
$\mathcal{Z}_{\intt}(\varepsilon_0)$.
Applying
Lemma \ref{Lemma-Large}
and
Lemma \ref{Lemma-Large-2},
we derive the corresponding
high-frequency estimates with suitable regularities of the initial data.
In contrast to the low-frequency region,
the polynomial growth disappears
in the high-frequency analysis with the decay rate $-\frac{1}{2}\left(\frac{1}{p}-\frac{1}{q}\right)$,
since the oscillatory multipliers
no longer contain
the singular low-frequency structure.

Combining the low-frequency estimates
and high-frequency estimates,
we directly complete the proof
of Theorem \ref{Theorem-Lq}.

\begin{remark}
The required regularities depend on whether the wave speeds are equal, because the denominator $\lambda_{\mathrm{I},1}^2-\lambda_{\mathrm{I},2}^2$ has different high-frequency behavior described in \eqref{Eq-01}.
More precisely, the equal speed configuration produces an additional cancellation of the leading high-frequency terms, which improves the regularity requirements.
\end{remark}

\section{Optimal large time $L^2$-growth estimates of solutions}
\setcounter{equation}{0}
\label{Section-L2-est}

\subsection{Sharp $L^2$ asymptotics of Fourier multipliers}

\hspace{5mm}
We begin with several asymptotic estimates for the dominant Fourier multipliers
appearing in the representations of solutions.
The following lemma shows that the quadratic oscillation generated by
$\frac{\sin(c_{\mathrm D}|\xi|^2t)}
{c_{\mathrm D}|\xi|^2}$
provides the leading low-frequency asymptotic profile.

\begin{lemma}\label{Lemma-Approx}
	Suppose that $\ell\in\{0,1\}$.
	Then, the following asymptotic estimate:
	\begin{align*}
		&\left\|\chi_{\intt}(\xi)|\xi|^{\ell}
		\left(\frac{\sin(\lambda_{\mathrm{I},2}t)}{\lambda_{\mathrm{I},2}}-\frac{\sin(c_{\mathrm{D}}|\xi|^2t)}{c_{\mathrm{D}}|\xi|^2}\right)
		\right\|_{L^2}^2
		=
		o(t^{\frac{3}{2}-\ell})
	\end{align*}
	holds as
	$t\to+\infty$.
\end{lemma}
\begin{proof}
	By the low-frequency expansion of
	$\lambda_{\mathrm I,2}$,
	we have
	\begin{align*}
		\chi_{\intt}(\xi)
		\left|
		\frac{1}{\lambda_{\mathrm I,2}}
		-
		\frac{1}{c_{\mathrm{D}}|\xi|^2}
		\right|
		=
		\chi_{\intt}(\xi)
		\left|
		\frac{O(|\xi|^6)}
		{c_{\mathrm{D}}|\xi|^2\lambda_{\mathrm I,2}}
		\right|
		\lesssim
		\chi_{\intt}(\xi)|\xi|^2,
	\end{align*}
which suggests
	\begin{align}\label{Triangle-01}
		&
		\left\|
		\chi_{\intt}(\xi)
		|\xi|^{\ell}
		\left(
		\frac{\sin(\lambda_{\mathrm{I},2}t)}
		{\lambda_{\mathrm{I},2}}
		-
		\frac{\sin(\lambda_{\mathrm{I},2}t)}
		{c_{\mathrm{D}}|\xi|^2}
		\right)
		\right\|_{L^2}^2
		\lesssim
		\left\|
		\chi_{\intt}(\xi)
		|\xi|^{\ell+2}\,
		\right\|_{L^2}^2
		\lesssim1.
	\end{align}
	
	Next,
	we approximate
	$\sin(\lambda_{\mathrm{I},2}t)$
	by
	$\sin(c_{\mathrm{D}}|\xi|^2t)$
	in the $L^2$ framework,
	rather than relying on pointwise estimates
	in the Fourier space.
	Indeed,
	the error estimate in the Fourier space
	\begin{align*}
		\chi_{\intt}(\xi)
		\left|
		\sin(\lambda_{\mathrm{I},2}t)
		-
		\sin(c_{\mathrm{D}}|\xi|^2t)
		\right|
		\lesssim
		\chi_{\intt}(\xi)
		|\xi|^4t
	\end{align*}
	produces the additional growth factor
	$t$,
	which cannot be compensated
	in the absence of the dissipative structure.
	This is substantially different from
	the dissipative Timoshenko system \eqref{Eq-dissipative-Timoshenko} studied in
	\cite[Inequality (3.18)]{Chen=2026},
	in which the damping factor $\mathrm{e}^{-\frac{\gamma}{2\rho}|\xi|^2t}$
	provides the additional decay rate (but fails when $\gamma=0$) as follows:
	\begin{align*}
	\chi_{\intt}(\xi)|\xi|^4t \,\mathrm{e}^{-\frac{\gamma}{2\rho}|\xi|^2t}\lesssim t^{-1}\chi_{\intt}(\xi) \,\mathrm{e}^{-c|\xi|^2t}.
	\end{align*}
	In the vanishing dissipation case, we introduce the scaled variable
	$
	\eta=\sqrt{t}\,\xi
	$
	to obtain
	\begin{align*}
		&
		\left\|
		\chi_{\intt}(\xi)
		|\xi|^{\ell}
		\left(
		\frac{\sin(\lambda_{\mathrm{I},2}t)}
		{c_{\mathrm{D}}|\xi|^2}
		-
		\frac{\sin(c_{\mathrm{D}}|\xi|^2t)}
		{c_{\mathrm{D}}|\xi|^2}
		\right)
		\right\|_{L^2}^2
		\\
		&=
		t^{\frac32-\ell}
		\int_{\mathbb R}
		\chi_{\intt}
		\left(
		\tfrac{\eta}{\sqrt t}
		\right)
		|\eta|^{2(\ell-2)}
		\left[
		\sin
		\left(
		|\eta|^2
		g_0
		\left(
		\tfrac{\eta}{\sqrt t}
		\right)
		\right)
		-
		\sin(c_{\mathrm D}|\eta|^2)
		\right]^2
		\mathrm d\eta,
	\end{align*}
	where
$
	\lambda_{\mathrm I,2}
	=
	|\xi|^2g_0(\xi)
$ with
	$g_0(0)=c_{\mathrm D}>0$. 
	As
	$t\to+\infty$,
	we have
	\[
	g_0
	\left(
	\tfrac{\eta}{\sqrt t}
	\right)
	\to
	g_0(0)=c_{\mathrm D}
	\]
	for every fixed
	$\eta$.
	Hence, the next convergence:
	\begin{align*}
		&
		\chi_{\intt}
		\left(
		\tfrac{\eta}{\sqrt t}
		\right)
		|\eta|^{2(\ell-2)}
		\left[
		\sin
		\left(
		|\eta|^2
		g_0
		\left(
		\tfrac{\eta}{\sqrt t}
		\right)
		\right)
		-
		\sin(c_{\mathrm D}|\eta|^2)
		\right]^2
		\to0
	\end{align*}
holds pointwise as $t\to+\infty$.
	By the elementary bounds
	$
	|\sin y|
	\leqslant
	|y|
$ and $
	|\sin y|
	\leqslant1$, 
	we derive
	\begin{align*}
		&
		\chi_{\intt}
		\left(
		\tfrac{\eta}{\sqrt t}
		\right)
		\left|
		|\eta|^{2(\ell-2)}
		\left[
		\sin
		\left(
		|\eta|^2
		g_0
		\left(
		\tfrac{\eta}{\sqrt t}
		\right)
		\right)
		-
		\sin(c_{\mathrm D}|\eta|^2)
		\right]^2
		\right|
		\lesssim
		\min
		\left\{
		|\eta|^{2\ell},
		|\eta|^{2(\ell-2)}
		\right\}.
	\end{align*}
	Since
	\begin{align*}
		\int_{\mathbb R}
		\min
		\left\{
		|\eta|^{2\ell},
		|\eta|^{2(\ell-2)}
		\right\}
		\mathrm d\eta
		<+\infty
	\end{align*}
	for
	$\ell\in\{0,1\}$,
	the dominated convergence theorem yields
	\begin{align}\label{Triangle-02}
		&
		\left\|
		\chi_{\intt}(\xi)
		|\xi|^{\ell}
		\left(
		\frac{\sin(\lambda_{\mathrm{I},2}t)}
		{c_{\mathrm{D}}|\xi|^2}
		-
		\frac{\sin(c_{\mathrm{D}}|\xi|^2t)}
		{c_{\mathrm{D}}|\xi|^2}
		\right)
		\right\|_{L^2}^2
		=
		o(t^{\frac32-\ell})
	\end{align}
	as
	$t\to+\infty$.	Finally,
	since
	$1=o(t^{\frac32-\ell})$
	for
	$\ell\in\{0,1\}$,
	combining
	\eqref{Triangle-01}
	and
	\eqref{Triangle-02}
	completes the proof immediately.
\end{proof}

The next optimal estimate reveals
the precise low-frequency scaling
responsible for the growth rates
$t^{\frac{3}{4}}$
and
$t^{\frac{1}{4}}$.

\begin{lemma}
	\label{Lemma-Kernel}
	
	Suppose that
	$\ell\in\{0,1\}$.
	Then,
	the Fourier multiplier satisfies
	the following sharp scaling asymptotics:
	\begin{align*}
		\left\|
		\chi_{\intt}(\xi)
		|\xi|^{\ell}\,
		\frac{
			\sin(c_{\mathrm D}|\xi|^2t)
		}{
			c_{\mathrm D}|\xi|^2
		}
		\right\|_{L^2}^2
		\approx
		t^{\frac{3}{2}-\ell}
	\end{align*}
	as
	$t\to+\infty$.
	
\end{lemma}

\begin{remark}
	The scaling $\eta=\sqrt{t}\,\xi$ 
	shows that the singular amplitude
	$|\xi|^{-2}$ combined with the quadratic oscillation
	$\sin(c_{\mathrm D}|\xi|^2t)$
	produces the growth factor
	$t^{\frac{3}{4}}$.
	Similarly,
	the additional factor
	$|\xi|$
	reduces the growth rate
	to
	$t^{\frac{1}{4}}$.
\end{remark}
\begin{proof}
	Concerning the upper bound, the approach in \cite[Lemma 4.1]{Chen=2026} is not applicable here
	due to the absence of the crucial exponentially decaying factor $\mathrm{e}^{-\frac{\gamma}{2\rho}|\xi|^2t}$,
	which produces an uncontrollable growth factor in time.
	To overcome this difficulty, we introduce two time-dependent separating lines
	$|\xi|=t^{-\alpha_{\ell}}$
	and
	$|\xi|=t^{-\beta_{\ell}}$,
	where
	\[
	\alpha_{\ell}:=\frac12\ \ \mbox{and}\ \ 
	\beta_{\ell}:=\frac{3-2\ell}{4(2-\ell)}\ \ \mbox{satisfying}\ \ \alpha_{\ell}>\beta_{\ell}>0.
	\]
	The first region captures the quadratic oscillatory scaling,
	whereas the third region is dominated by the singular weight, and the intermediate region balances these two mechanisms.
	Consequently, we implement the last philosophy to arrive at
	\begin{align*}
		&\left\|
		\chi_{\intt}(\xi)
		|\xi|^{\ell}\,
		\frac{
			\sin(c_{\mathrm D}|\xi|^2t)
		}{
			c_{\mathrm D}|\xi|^2
		}
		\right\|_{L^2}^2\\
		&\lesssim
		\left(
		\int_{|\xi|\leqslant t^{-\alpha_{\ell}}}
		+
		\int_{t^{-\alpha_{\ell}}\leqslant|\xi|\leqslant t^{-\beta_{\ell}}}
		+
		\int_{t^{-\beta_{\ell}}\leqslant|\xi|\leqslant\varepsilon_0}
		\right)
		\frac{
			|\sin(c_{\mathrm D}|\xi|^2t)|^2
		}{
			|\xi|^{4-2\ell}
		}
		\,\mathrm d\xi
		\\
		&\lesssim
		t^2
		\int_{|\xi|\leqslant t^{-\alpha_{\ell}}}
		\left|
		\frac{
			\sin(c_{\mathrm D}|\xi|^2t)
		}{
			c_{\mathrm D}|\xi|^2t
		}
		\right|^2
		|\xi|^{2\ell}
		\,\mathrm d\xi
		+
		\int_{t^{-\alpha_{\ell}}\leqslant|\xi|\leqslant t^{-\beta_{\ell}}}
		|\xi|^{-4+2\ell}
		\,\mathrm d\xi
		+
		t^{2\beta_{\ell}(2-\ell)}
		\int_{t^{-\beta_{\ell}}\leqslant|\xi|\leqslant\varepsilon_0}
		\mathrm d\xi
		\\
		&\lesssim
		t^2
		\int_0^{t^{-\alpha_{\ell}}}
		|\xi|^{2\ell}
		\,\mathrm d|\xi|
		+
		\int_{t^{-\alpha_{\ell}}}^{t^{-\beta_{\ell}}}
		|\xi|^{-4+2\ell}
		\,\mathrm d|\xi|
		+
		t^{2\beta_{\ell}(2-\ell)}
		\int_{t^{-\beta_{\ell}}}^{\varepsilon_0}
		\mathrm d|\xi|
		\\
		&\lesssim
		t^{2-\alpha_{\ell}(1+2\ell)}
		+
		t^{\alpha_{\ell}(3-2\ell)}
		+
		t^{2\beta_{\ell}(2-\ell)}.
	\end{align*}
	The choices of
	$\alpha_{\ell}$
	and
	$\beta_{\ell}$
	are made so that
	all three contributions
	have the same asymptotic order
	$t^{\frac32-\ell}$, in other words,
	\[
	t^{2-\alpha_{\ell}(1+2\ell)}+t^{\alpha_{\ell}(3-2\ell)}+t^{2\beta_{\ell}(2-\ell)}
	\approx t^{\frac32-\ell},
	\]
	which proves the desired upper bound.
	
	Turning to the lower bound, we shrink the integration region
	to the dominant low-frequency zone
	$\{|\xi|\leqslant\delta t^{-\frac12}\}$,
	where
	$\delta>0$
	is sufficiently small.
	Indeed,
	since
	$
	\frac{\sin y}{y}\to1
	$ as $y\to0$,
	choosing
	$\delta$
	sufficiently small,
	we guarantee
	\[
	\left|
	\frac{
		\sin(c_{\mathrm D}|\xi|^2t)
	}{
		c_{\mathrm D}|\xi|^2t
	}
	\right|
	\geqslant
	\frac{1}{\sqrt2}.
	\]
	It follows that
	\begin{align*}
		\left\|\chi_{\intt}(\xi)|\xi|^{\ell}\,\frac{\sin(c_{\mathrm D}|\xi|^2t)}{c_{\mathrm D}|\xi|^2}\right\|_{L^2}^2
		&\gtrsim t^2
		\int_{|\xi|\leqslant\delta t^{-\frac12}}\left|\frac{\sin(c_{\mathrm D}|\xi|^2t)}{c_{\mathrm D}|\xi|^2t}\right|^2|\xi|^{2\ell}\,\mathrm d\xi\\
		&\gtrsim t^2\int_0^{\delta t^{-\frac12}}|\xi|^{2\ell}\,\mathrm d|\xi|\\
		&\gtrsim t^{\frac32-\ell}.
	\end{align*}
	This completes the proof.
\end{proof}

\subsection{$L^2$ approximations of kernels}

\hspace{5mm}
Following the asymptotic approach in
\cite{Takeda=2026},
we compare the exact Fourier multipliers
\begin{align}\label{Notation-K-ell}
	\mathcal K_0^{\ell}(t,x)
	:=
	\chi_{\intt}(D)
	\mathcal F^{-1}_{\xi\to x}
	\left(
	(i\xi)^{\ell}\,
	\frac{\sin(c_{\mathrm{D}}|\xi|^2t)}
	{c_{\mathrm D}|\xi|^2}
	\right)
\end{align}
acting on the initial data (in the sense of convolution)
with the corresponding asymptotic kernels
multiplied by the moments
of the initial data.

\begin{lemma}\label{Lemma-Approx-2}
	Suppose that
	$f\in L^1$
	and
	$\ell\in\{0,1\}$.
	Then,
	the following approximation:
	\begin{align*}
		\left\|
		\mathcal K_0^{\ell}(t,\cdot)\ast_{(x)}f(\cdot)
		-
		\mathcal K_0^{\ell}(t,\cdot)P_f
		\right\|_{L^2}^2
		=
		o(t^{\frac{3}{2}-\ell})
	\end{align*}
	holds as
	$t\to+\infty$.
\end{lemma}

\begin{proof}
	Applying the mean value theorem
	to the spatial variable,
	we have
	\[
	|\mathcal K_0^{\ell}(t,x-y)
	-
	\mathcal K_0^{\ell}(t,x)|
	\lesssim
	|y|
	\,|\partial_x\mathcal K_0^{\ell}(t,x-\eta_0y)|
	\]
	for some
	$\eta_0\in(0,1)$. We separate it into two parts
	\begin{align*}
		&\left\|
		\mathcal K_0^{\ell}(t,\cdot)\ast_{(x)}f(\cdot)
		-
		\mathcal K_0^{\ell}(t,\cdot)P_f
		\right\|_{L^2}^2\\
		&\leqslant
		\left\|
		\int_{|y|\leqslant t^{\frac18}}
		\big(
		\mathcal K_0^{\ell}(t,\cdot-y)
		-
		\mathcal K_0^{\ell}(t,\cdot)
		\big)
		f(y)\,\mathrm dy
		\right\|_{L^2}^2
		+
		\left\|
		\int_{|y|\geqslant t^{\frac18}}
		\big(
		|\mathcal K_0^{\ell}(t,\cdot-y)|
		+
		|\mathcal K_0^{\ell}(t,\cdot)|
		\big)
		|f(y)|\,\mathrm dy
		\right\|_{L^2}^2.
	\end{align*}
	The time-dependent threshold
	$|y|=t^{\frac18}$
	is chosen so that
	the near-field contribution
	remains asymptotically negligible
	compared with
	the leading order
	$t^{\frac32-\ell}$.
	Using also Plancherel's identity,
	we arrive at
	\begin{align*}
		&\left\|
		\mathcal K_0^{\ell}(t,\cdot)\ast_{(x)}f(\cdot)
		-
		\mathcal K_0^{\ell}(t,\cdot)P_f
		\right\|_{L^2}^2\\
		&\lesssim
		t^{\frac14}
		\left\|
		\chi_{\intt}(\xi)
		(i\xi)^{\ell+1}\,
		\frac{\sin(c_{\mathrm{D}}|\xi|^2t)}
		{c_{\mathrm{D}}|\xi|^2}
		\right\|_{L^2}^2
		\|f\|_{L^1}^2
		+
		\left\|
		\chi_{\intt}(\xi)
		|\xi|^{\ell}\,
		\frac{\sin(c_{\mathrm{D}}|\xi|^2t)}
		{c_{\mathrm{D}}|\xi|^2}
		\right\|_{L^2}^2
		\|f\|_{L^1(|x|\geqslant t^{\frac18})}^2
		\\
		&\lesssim
		t^{\frac{3-2\ell}{4}}
		\|f\|_{L^1}^2
		+
		o(t^{\frac32-\ell}),
	\end{align*}
	as $t\to+\infty$, thanks to Lemma \ref{Lemma-Kernel}.
	Since our assumption $f\in L^1$,
	we have
	\[
	\lim_{t\to+\infty}
	\int_{|x|\geqslant t^{\frac18}}
	|f(x)|\,\mathrm dx
	=
	0.
	\]
	Moreover,
	$
	\frac{3-2\ell}{4}
	<
	\frac32-\ell
	$
	for all
	$\ell\in\{0,1\}$,
	so that the first term
	is lower-order.
	This completes the proof directly.
\end{proof}

\begin{remark}
	Actually, Lemma \ref{Lemma-Approx-2}
	shows that
	the leading large time behavior
	is completely determined
	by the zeroth moment
	$P_f$.
	This is the fundamental mechanism
	behind the asymptotic profiles
	derived later for
	$\varphi$
	and
	$\psi$.
\end{remark}

The next result shows that
when the zeroth moment vanishes,
the leading asymptotic behavior
is determined by the first moment.

\begin{lemma}\label{Lemma-Approx-3}
	Suppose that
	$f\in L^{1,1}$
	such that
	$P_f=0$.
	Then,
	the following approximation:
	\begin{align*}
		\left\|
		\mathcal K_0^{0}(t,\cdot)\ast_{(x)}f(\cdot)
		-
		\partial_x\mathcal K_0^0(t,\cdot)M_f
		\right\|_{L^2}^2
		=
		o(t^{\frac{1}{2}})
	\end{align*}
	holds as
	$t\to+\infty$.
\end{lemma}

\begin{proof}
	Recalling the notation in
	\eqref{Notation-K-ell}
	and applying the mean value theorem,
	we have
	\[
	|\mathcal K_0^{0}(t,x-y)
	-
	\mathcal K_0^{0}(t,x)
	-(-y)
	\partial_x\mathcal K_0^0(t,x)|
	\lesssim
	|y|^2
	|\partial_x^2\mathcal K_0^{0}(t,x-\eta_1y)|
	\]
	for some $\eta_1\in(0,1)$.
	Therefore, using $\mathcal K_0^{0}(t,x)P_f=0$ from $P_f=0$, we estimate
	\begin{align*}
		&
		\left\|
		\mathcal K_0^{0}(t,\cdot)\ast_{(x)}f(\cdot)
		-
		\partial_x\mathcal K_0^0(t,\cdot)M_f
		\right\|_{L^2}^2
		\\
		&\leqslant
		\left\|
		\int_{|y|\leqslant t^{\frac{1}{16}}}
		\big(
		\mathcal K_0^{0}(t,\cdot-y)
		-
		\mathcal K_0^{0}(t,\cdot)
		-(-y)\partial_x\mathcal K_0^0(t,\cdot)
		\big)
		f(y)\,\mathrm dy
		\right\|_{L^2}^2
		\\
		&\quad
		+
		\left\|
		\int_{|y|\geqslant t^{\frac{1}{16}}}
		\big(
		|\mathcal K_0^{0}(t,\cdot-y)
		-
		\mathcal K_0^{0}(t,\cdot)|
		+
		|y\partial_x\mathcal K_0^{0}(t,\cdot)|
		\big)
		|f(y)|\,\mathrm dy
		\right\|_{L^2}^2.
	\end{align*}
	The time-dependent threshold
	$|y|=t^{\frac{1}{16}}$
	is chosen so that
	the near-field contribution
	remains asymptotically negligible
	compared with
	the leading order
	$t^{\frac12}$.
	Using also Plancherel's identity,
	we may deduce
	\begin{align*}
		&\left\|
		\mathcal K_0^{0}(t,\cdot)\ast_{(x)}f(\cdot)
		-
		\partial_x\mathcal K_0^0(t,\cdot)M_f
		\right\|_{L^2}^2\\
		&\lesssim
		t^{\frac14}
		\left\|
		\chi_{\intt}(\xi)
		\sin(c_{\mathrm{D}}|\xi|^2t)
		\right\|_{L^2}^2
		\|f\|_{L^1}^2
		+
		\left\|
		\chi_{\intt}(\xi)
		i\xi
		\frac{\sin(c_{\mathrm{D}}|\xi|^2t)}
		{c_{\mathrm{D}}|\xi|^2}
		\right\|_{L^2}^2
		\|xf\|_{L^1(|x|\geqslant t^{\frac{1}{16}})}^2
		\\
		&\lesssim
		t^{\frac{1}{4}}
		\|f\|_{L^1}^2
		+
		o(t^{\frac{1}{2}}),
	\end{align*}
	as $t\to+\infty$, thanks to Lemma \ref{Lemma-Kernel}.
	Here the singular factor $|\xi|^{-2}$ disappears after differentiation, which reduces the growth order.
	Since our assumption $f\in L^{1,1}$, we have
	\[
	\lim_{t\to+\infty} \int_{|x|\geqslant t^{\frac{1}{16}}} |xf(x)|\,\mathrm dx = 0.
	\]
	This completes the proof.
\end{proof}

The combination of Lemma \ref{Lemma-Approx-2} with $\ell=1$ and Lemma \ref{Lemma-Approx-3} gives the following final auxiliary lemma by applying the triangle inequality straightforwardly.

\begin{lemma}\label{Lemma-Approx-4}
	Suppose that
	$f_0\in L^{1,1}$
	and
	$f_1\in L^1$
	such that
	$P_{f_0}=0$.
	Then,
	the following approximation:
	\begin{align*}
		\left\|
		\mathcal K_0^0(t,\cdot)\ast_{(x)}f_0(\cdot)
		+
		\mathcal K_0^1(t,\cdot)\ast_{(x)}f_1(\cdot)
		-
		\mathcal K_0^1(t,\cdot)
		\big(M_{f_0}+P_{f_1}\big)
		\right\|_{L^2}^2
		=
		o(t^{\frac{1}{2}})
	\end{align*}
	holds as
	$t\to+\infty$.
\end{lemma}

\subsection{Upper bound estimates for large time}

\hspace{5mm}
The leading large time behavior of solutions is generated by the low-frequency sine multipliers
associated with the initial velocity
$\varphi_1$.
In particular,
the low-frequency cosine multipliers
produce only lower-order contributions. 

We first extract
the dominant low-frequency contributions
from the representations of solutions.
Obviously, they fulfill the following pointwise estimates in the Fourier space:
\begin{align*}
	\chi_{\intt}(\xi)
	\left(
	|\widehat{K}_{\varphi,0}(t,|\xi|)|
	+
	|\widehat{K}_{\psi,0}(t,|\xi|)|
	\right)
	&\lesssim
	\chi_{\intt}(\xi),
	\\
	\chi_{\intt}(\xi)
	|\widehat{K}_{\psi,1}(t,|\xi|)|
	&\lesssim
	\chi_{\intt}(\xi)
	|\xi|
	\left|
	\frac{\sin(\lambda_{\mathrm I,2}t)}
	{\lambda_{\mathrm I,2}}
	\right|,
	\\
	\chi_{\intt}(\xi)
	|\widehat{K}_{\varphi,1}(t,|\xi|)|
	&\lesssim
	\chi_{\intt}(\xi)
	\left|
	\frac{\sin(\lambda_{\mathrm I,2}t)}
	{\lambda_{\mathrm I,2}}
	\right|.
\end{align*}
For the low-frequency analysis,
we occasionally employ
the $H^{-1}$ framework
instead of the $L^2$ framework
in order to match
the high-frequency regularity structure (to be explained later).
Indeed,
\begin{align*}
	\|\chi_{\intt}(\xi)\widehat{f}(\xi)\|_{L^2}^2
	&\lesssim
	\|\chi_{\intt}(\xi)\langle\xi\rangle\|_{L^{\infty}}^2
	\|\langle\xi\rangle^{-1}\widehat{f}(\xi)\|_{L^2}^2
	\\
	&\lesssim
	\|f\|_{H^{-1}}^2.
\end{align*}
Applying
Lemma \ref{Lemma-Ikehata}
and the Hausdorff--Young inequality,
we obtain
\begin{align*}
	\|\varphi(t,\cdot)\|_{L^2_{\chi}}^2
	&\lesssim
	\|\varphi_0\|_{L^2}^2
	+
	\left\|
	\chi_{\intt}(\xi)
	\left|
	\frac{\sin(\lambda_{\mathrm I,2}t)}
	{\lambda_{\mathrm I,2}}
	\right|
	\big(
	|P_{\varphi_1}|
	+
	|\xi|\,\|\varphi_1\|_{L^{1,1}}
	\big)
	\right\|_{L^2}^2
	\\
	&\quad
	+
	\|\psi_0\|_{H^{-1}}^2
	+
	\left\|
	\chi_{\intt}(\xi)
	|\xi|
	\left|
	\frac{\sin(\lambda_{\mathrm I,2}t)}
	{\lambda_{\mathrm I,2}}
	\right|
	\right\|_{L^2}^2
	\|\psi_1\|_{L^1}^2.
\end{align*}
By Lemma \ref{Lemma-Approx} and Lemma \ref{Lemma-Kernel}, we further deduce that
\begin{align*}
	\|\varphi(t,\cdot)\|_{L^2_{\chi}}^2
	&\lesssim
	\|\varphi_0\|_{L^2}^2
	+
	\left\|
	\chi_{\intt}(\xi)
	\left|
	\frac{\sin(c_{\mathrm{D}} |\xi|^2t)}
	{c_{\mathrm{D}} |\xi|^2}
	\right|
	\right\|_{L^2}^2
	|P_{\varphi_1}|^2
	+
	o(t^{\frac{3}{2}})
	|P_{\varphi_1}|^2
	+
	\|\psi_0\|_{H^{-1}}^2
	\\
	&\quad
	+
	\left\|
	\chi_{\intt}(\xi)
	|\xi|
	\left|
	\frac{\sin(c_{\mathrm{D}} |\xi|^2t)}
	{c_{\mathrm{D}} |\xi|^2}
	\right|
	\right\|_{L^2}^2
	\left(
	\|\varphi_1\|_{L^{1,1}}^2
	+
	\|\psi_1\|_{L^1}^2
	\right)
	+
	o(t^{\frac{1}{2}})
	\left(
	\|\varphi_1\|_{L^{1,1}}^2
	+
	\|\psi_1\|_{L^1}^2
	\right)
	\\
	&\lesssim
	\|\varphi_0\|_{L^2}^2
	+
	t^{\frac{1}{2}}
	\left(
	t\,|P_{\varphi_1}|^2
	+
	\|\varphi_1\|_{L^{1,1}}^2
	\right)
	+
	\|\psi_0\|_{H^{-1}}^2
	+
	t^{\frac{1}{2}}
	\|\psi_1\|_{L^1}^2
\end{align*}
as
$t\to+\infty$.
Furthermore,
a direct subtraction implies
\begin{align*}
	&
	\left\|
	\varphi(t,\cdot)
	-
	\mathcal{F}^{-1}_{\xi\to x}
	\left(
	\frac{\lambda_{\mathrm I,1}^2}
	{\lambda_{\mathrm I,1}^2-\lambda_{\mathrm I,2}^2}
	\,
	\frac{\sin(\lambda_{\mathrm I,2}t)}
	{\lambda_{\mathrm I,2}}
	\right)
	\ast_{(x)}
	\varphi_1(\cdot)
	\right\|_{L^2_{\chi}}^2
\lesssim
	\|(\varphi_0,\varphi_1)\|_{L^2\times L^2}^2
	+
	t^{\frac{1}{2}}
	\|(\psi_0,\psi_1)\|_{H^{-1}\times L^2}^2.
\end{align*}
Here we used that
\[
\frac{\lambda_{\mathrm I,1}^2}
{\lambda_{\mathrm I,1}^2-\lambda_{\mathrm I,2}^2}
=
1+O(|\xi|^2)
\]
for low-frequencies.
Motivated by
Lemma \ref{Lemma-Approx},
and applying
Lemma \ref{Lemma-Approx-2}
with
$\ell=0$,
we obtain
\begin{align*}
	&
	\left\|
	\mathcal{F}^{-1}_{\xi\to x}
	\left(
	\frac{\lambda_{\mathrm I,1}^2}
	{\lambda_{\mathrm I,1}^2-\lambda_{\mathrm I,2}^2}
	\,
	\frac{\sin(\lambda_{\mathrm I,2}t)}
	{\lambda_{\mathrm I,2}}
	\right)
	\ast_{(x)}
	\varphi_1(\cdot)
	-
	\mathcal{F}^{-1}_{\xi\to x}
	\left(
	\frac{\sin(c_{\mathrm D}|\xi|^2t)}
	{c_{\mathrm D}|\xi|^2}
	\right)
	P_{\varphi_1}
	\right\|_{L^2_{\chi}}^2
	=
	o(t^{\frac32}),
\end{align*}
which yields the asymptotic approximation
\begin{align}\label{Est-02}
	\left\|
	\varphi(t,\cdot)
	-
	\mathcal{F}^{-1}_{\xi\to x}
	\left(
	\frac{\sin(c_{\mathrm D}|\xi|^2t)}
	{c_{\mathrm D}|\xi|^2}
	\right)
	P_{\varphi_1}
	\right\|_{L^2_{\chi}}^2
	=
	o(t^{\frac32})
\end{align}
as
$t\to+\infty$.

Analogously,
thanks to
\begin{align*}
	\chi_{\intt}(\xi)
	\left(
	|\widehat{G}_{\varphi,0}(t,|\xi|)|
	+
	|\widehat{G}_{\psi,0}(t,|\xi|)|
	+
	|\widehat{G}_{\psi,1}(t,|\xi|)|
	\right)
	&\lesssim
	\chi_{\intt}(\xi),
	\\
	\chi_{\intt}(\xi)
	|\widehat{G}_{\varphi,1}(t,|\xi|)|
	&\lesssim
	\chi_{\intt}(\xi)
	|\xi|
	\left|
	\frac{\sin(\lambda_{\mathrm I,2}t)}
	{\lambda_{\mathrm I,2}}
	\right|,
\end{align*}
we conclude
\begin{align*}
	\|\psi(t,\cdot)\|_{L^2_{\chi}}^2
	\lesssim
	\|\varphi_0\|_{H^{-1}}^2
	+
	\left(
	t^{\frac{1}{2}}
	|P_{\varphi_1}|^2
	+
	\|\varphi_1\|_{L^{1,1}}^2
	\right)
	+
	\|\psi_0\|_{L^2}^2
	+
	\|\psi_1\|_{H^{-1}}^2
\end{align*}
as
$t\to+\infty$.
Moreover,
applying
Lemma \ref{Lemma-Approx-2}
with
$\ell=1$,
we derive
\begin{align}\label{Est-03}
	\left\|
	\psi(t,\cdot)
	-
	\mathcal{F}^{-1}_{\xi\to x}
	\left(
	i\xi
	\frac{\sin(c_{\mathrm D}|\xi|^2t)}
	{c_{\mathrm D}|\xi|^2}
	\right)
	P_{\varphi_1}
	\right\|_{L^2_{\chi}}^2
	=
	o(t^{\frac12})
\end{align}
as
$t\to+\infty$.

We next employ
the asymptotic expansions
of characteristic roots
for
$|\xi|\gg1$.
Here,
the high-frequency behavior depends on whether
the wave speeds are equal.
Moreover,
thanks to the separation property
\[
\lambda_{\mathrm{I},1}\neq\lambda_{\mathrm{I},2} \ \ \mbox{for} \ \ \xi\in\mathcal Z_{\bdd}(\varepsilon_0,N_0),
\]
the coefficients appearing
in the representations
of
$\widehat{\varphi}$
and
$\widehat{\psi}$
remain uniformly bounded
away from singularities.

\begin{itemize}
	
	\item
	If
	$c_{\mathrm S}\neq c_{\mathrm R}$,
	then
	\begin{align*}
		\big(1-\chi_{\intt}(\xi)\big)
		|\widehat{\varphi}|
		&\lesssim
		\big(1-\chi_{\intt}(\xi)\big)
		\left(
		|\widehat{\varphi}_0|
		+
		\langle\xi\rangle^{-1}
		|\widehat{\varphi}_1|
		+
		\langle\xi\rangle^{-1}
		|\widehat{\psi}_0|
		+
		\langle\xi\rangle^{-2}
		|\widehat{\psi}_1|
		\right),
		\\
		\big(1-\chi_{\intt}(\xi)\big)
		|\widehat{\psi}|
		&\lesssim
		\big(1-\chi_{\intt}(\xi)\big)
		\left(
		\langle\xi\rangle^{-1}
		|\widehat{\varphi}_0|
		+
		\langle\xi\rangle^{-2}
		|\widehat{\varphi}_1|
		+
		|\widehat{\psi}_0|
		+
		\langle\xi\rangle^{-1}
		|\widehat{\psi}_1|
		\right).
	\end{align*}
	
	Consequently, Plancherel's identity gives
	\begin{align*}
		\|\varphi(t,\cdot)\|_{L^2_{1-\chi}}^2
		&\lesssim
		\|(\varphi_0,\varphi_1)\|_{L^2\times H^{-1}}^2
		+
		\|(\psi_0,\psi_1)\|_{H^{-1}\times H^{-2}}^2,
		\\
		\|\psi(t,\cdot)\|_{L^2_{1-\chi}}^2
		&\lesssim
		\|(\varphi_0,\varphi_1)\|_{H^{-1}\times H^{-2}}^2
		+
		\|(\psi_0,\psi_1)\|_{L^2\times H^{-1}}^2.
	\end{align*}
	
	\item
	If
	$c_{\mathrm S}=c_{\mathrm R}$,
	then
	\begin{align*}
		\big(1-\chi_{\intt}(\xi)\big)
		|\widehat{\varphi}|
		+
		\big(1-\chi_{\intt}(\xi)\big)
		|\widehat{\psi}|\lesssim
		\big(1-\chi_{\intt}(\xi)\big)
		\left(
		|\widehat{\varphi}_0|
		+
		\langle\xi\rangle^{-1}
		|\widehat{\varphi}_1|
		+
		|\widehat{\psi}_0|
		+
		\langle\xi\rangle^{-1}
		|\widehat{\psi}_1|
		\right).
	\end{align*}
	
	Consequently, Plancherel's identity gives
	\begin{align*}
		\|\varphi(t,\cdot)\|_{L^2_{1-\chi}}^2
		+
		\|\psi(t,\cdot)\|_{L^2_{1-\chi}}^2
		\lesssim
		\|(\varphi_0,\varphi_1)\|_{L^2\times H^{-1}}^2
		+
		\|(\psi_0,\psi_1)\|_{L^2\times H^{-1}}^2.
	\end{align*}
	
\end{itemize}
In both cases,
no additional time growth
appears in the high-frequency zone.
Therefore,
the intrinsic large time growth mechanism
is generated entirely
by the low-frequency singular structure.
Moreover, the bounded property of sine functions suggests
\begin{align*}
	\left\|
	\mathcal{F}^{-1}_{\xi\to x}
	\left(
	(i\xi)^{\ell}\,
	\frac{\sin(c_{\mathrm D}|\xi|^2t)}
	{c_{\mathrm D}|\xi|^2}
	\right)
	P_{\varphi_1}
	\right\|_{L^2_{1-\chi}}^2
	&\lesssim
	\int_{|\xi|\geqslant\varepsilon_0}
	|\xi|^{-4+2\ell}
	\,\mathrm d\xi
	\,|P_{\varphi_1}|^2\\
	&\lesssim
	|P_{\varphi_1}|^2
\end{align*}
for all $\ell\in\{0,1\}$.
Hence,
the asymptotic profiles
obtained in
\eqref{Est-02}
and
\eqref{Est-03}
are generated exclusively
by the low-frequency region.

Applying Plancherel's identity
and combining all previous estimates,
we complete the proof
of the desired upper bound estimates.

\subsection{Lower bound estimates for large time}

\hspace{5mm}
For
$u\in\{\varphi,\psi\}$,
 Plancherel's identity implies
\[
\|u(t,\cdot)\|_{L^2}^2=\|\widehat{u}(t,\xi)\|_{L^2}^2\geqslant\|\chi_{\intt}(\xi)\widehat{u}(t,\xi)\|_{L^2}^2,
\]
which allows us to restrict the analysis to the low-frequency zone.

\subsubsection{Nontrivial zeroth moment
	$P_{\varphi_1}\neq0$}

\hspace{5mm}
Applying the triangle inequality, together with the estimate \eqref{Est-02} in the Fourier space and Lemma \ref{Lemma-Kernel}, we obtain
\begin{align*}
	\|\varphi(t,\cdot)\|_{L^2}^2
	&\gtrsim
	\left\|
	\chi_{\intt}(\xi)
	\frac{\sin(c_{\mathrm D}|\xi|^2t)}
	{c_{\mathrm D}|\xi|^2}
	P_{\varphi_1}
	\right\|_{L^2}^2
	-
	\left\|
	\chi_{\intt}(\xi)
	\left(
	\widehat{\varphi}(t,\xi)
	-
	\frac{\sin(c_{\mathrm D}|\xi|^2t)}
	{c_{\mathrm D}|\xi|^2}
	P_{\varphi_1}
	\right)
	\right\|_{L^2}^2
	\\
	&\gtrsim
	t^{\frac{3}{2}}
	|P_{\varphi_1}|^2
	-
	o(t^{\frac{3}{2}})
\end{align*}
as
$t\to+\infty$.

Similarly, combining the estimate \eqref{Est-03} in the Fourier space with Lemma \ref{Lemma-Kernel}, we derive
\begin{align*}
	\|\psi(t,\cdot)\|_{L^2}^2
	&\gtrsim
	\left\|
	\chi_{\intt}(\xi)
	i\xi
	\frac{\sin(c_{\mathrm D}|\xi|^2t)}
	{c_{\mathrm D}|\xi|^2}
	P_{\varphi_1}
	\right\|_{L^2}^2
	-
	\left\|
	\chi_{\intt}(\xi)
	\left(
	\widehat{\psi}(t,\xi)
	-
	i\xi
	\frac{\sin(c_{\mathrm D}|\xi|^2t)}
	{c_{\mathrm D}|\xi|^2}
	P_{\varphi_1}
	\right)
	\right\|_{L^2}^2
	\\
	&\gtrsim
	t^{\frac{1}{2}}
	|P_{\varphi_1}|^2
	-
	o(t^{\frac{1}{2}})
\end{align*}
as
$t\to+\infty$.

Therefore,
for sufficiently large time,
the leading profile terms dominate
the remainder contributions,
which completes the proof
of the lower bound estimates
under the condition
$P_{\varphi_1}\neq0$.

\subsubsection{Trivial zeroth moment
	$P_{\varphi_1}=0$}

\hspace{5mm}
Combining
\eqref{Rep-varphi}
with the low-frequency asymptotic expansions
of the characteristic roots,
we obtain immediately
\begin{align*}
	&
	\left\|
	\chi_{\intt}(\xi)
	\left(
	\widehat{\varphi}(t,\xi)
	-
	\frac{
		\lambda_{\mathrm{I},1}^2\,\widehat{\varphi}_1(\xi)
		-
		\frac{K}{\rho}i\xi\widehat{\psi}_1(\xi)
	}{
		\lambda_{\mathrm{I},1}^2
		-
		\lambda_{\mathrm{I},2}^2
	}
	\,
	\frac{
		\sin(\lambda_{\mathrm{I},2}t)
	}{
		\lambda_{\mathrm{I},2}
	}
	\right)
	\right\|_{L^2}^2
\lesssim
	\|(\varphi_0,\varphi_1)\|_{L^2\times H^{-1}}^2
	+
	\|(\psi_0,\psi_1)\|_{H^{-1}\times H^{-2}}^2.
\end{align*}
Consequently, the triangle inequality associated with the last estimate shows
\begin{align*}
	\|\varphi(t,\cdot)\|_{L^2_{\chi}}^2
	\gtrsim
	J_1(t)
	-
	\left(
	\|(\varphi_0,\varphi_1)\|_{L^2\times H^{-1}}^2
	+
	\|(\psi_0,\psi_1)\|_{H^{-1}\times H^{-2}}^2
	\right),
\end{align*}
in which we took the notion
\begin{align*}
	J_1(t):=
	\left\|\chi_{\intt}(\xi)\frac{\lambda_{\mathrm{I},1}^2\,\widehat{\varphi}_1(\xi)-\frac{K}{\rho}i\xi\widehat{\psi}_1(\xi)}{\lambda_{\mathrm{I},1}^2-\lambda_{\mathrm{I},2}^2}\,\frac{\sin(\lambda_{\mathrm{I},2}t)}{\lambda_{\mathrm{I},2}}\right\|_{L^2}^2.
\end{align*}

It remains to estimate lower bounds of $J_1(t)$ for large time.
Using the low-frequency asymptotic expansions of the characteristic roots, we may rewrite
\begin{align*}
	J_1(t)
	=
	\left\|
	\chi_{\intt}(\xi)
	\frac{
		\big(c_{\mathrm O}^2+O(|\xi|^2)\big)
		\widehat{\varphi}_1(\xi)
		-
		\frac{K}{\rho}i\xi\widehat{\psi}_1(\xi)
	}{
		c_{\mathrm O}^2+O(|\xi|^2)
	}
	\,
	\frac{
		\sin(\lambda_{\mathrm{I},2}t)
	}{
		\lambda_{\mathrm{I},2}
	}
	\right\|_{L^2}^2.
\end{align*}
Applying the triangle inequality again,
we further extract the following dominant contribution:
\begin{align*}
	J_1(t)
	&\gtrsim
	\left\|
	\chi_{\intt}(\xi)\left(\widehat{\varphi}_1(\xi)-\frac{K}{\rho\,c_{\mathrm O}^2}i\xi\widehat{\psi}_1(\xi)\right)
	\frac{\sin(c_{\mathrm D}|\xi|^2t)}{c_{\mathrm D}|\xi|^2}
	\right\|_{L^2}^2\\
	&\quad -
	\left\|\chi_{\intt}(\xi)|\xi|^2
	\left(|\widehat{\varphi}_1(\xi)|+|\xi|\,|\widehat{\psi}_1(\xi)|\right)
	\frac{\sin(\lambda_{\mathrm{I},2}t)}{\lambda_{\mathrm{I},2}}
	\right\|_{L^2}^2
	\\
	&\quad
	-
	\left\|
	\chi_{\intt}(\xi)
	\left(
	|\widehat{\varphi}_1(\xi)|
	+
	|\xi|\,
	|\widehat{\psi}_1(\xi)|
	\right)
	\left(
	\frac{
		\sin(\lambda_{\mathrm{I},2}t)
	}{
		\lambda_{\mathrm{I},2}
	}
	-
	\frac{
		\sin(c_{\mathrm D}|\xi|^2t)
	}{
		c_{\mathrm D}|\xi|^2
	}
	\right)
	\right\|_{L^2}^2,
\end{align*}
where we used
\begin{align*}
	&\chi_{\intt}(\xi)\left|\frac{\big(c_{\mathrm O}^2+O(|\xi|^2)\big)\widehat{\varphi}_1(\xi)-\frac{K}{\rho}i\xi\widehat{\psi}_1(\xi)}{c_{\mathrm O}^2+O(|\xi|^2)}-\left(\widehat{\varphi}_1(\xi)-\frac{K}{\rho\, c_{\mathrm O}^2}i\xi\widehat{\psi}_1(\xi)\right)\right|\\
&\lesssim\chi_{\intt}(\xi)|\xi|^2\left(|\widehat{\varphi}_1(\xi)|+|\xi|\,|\widehat{\psi}_1(\xi)|\right).
\end{align*}
Since $\frac{K}{\rho\,c_{\mathrm O}^2}=\frac{I_{\rho}}{\rho}$,
Lemma \ref{Lemma-Ikehata} with $P_{\varphi_1}=0$ and Lemma \ref{Lemma-Approx} with $\ell=1$ yield
\begin{align*}
	J_1(t)
	&\gtrsim
	\left\|
	\chi_{\intt}(\xi)
	\left(
	\widehat{\varphi}_1(\xi)
	-
	\frac{I_{\rho}}{\rho}
	i\xi\widehat{\psi}_1(\xi)
	\right)
	\frac{
		\sin(c_{\mathrm D}|\xi|^2t)
	}{
		c_{\mathrm D}|\xi|^2
	}
	\right\|_{L^2}^2
	-
	\left(
	\|\varphi_1\|_{L^{1,1}}^2
	+
	\|\psi_1\|_{L^1}^2
	\right)
	-
	o(t^{\frac{1}{2}}).
\end{align*}
Applying
Lemma \ref{Lemma-Approx-4}
with
$f_0=\varphi_1$
and
$f_1=\frac{I_{\rho}}{\rho}\psi_1$,
we further obtain
\begin{align*}
	J_1(t)
	&\gtrsim
	\left\|
	\mathcal F^{-1}_{\xi\to x}
	\left(
	i\xi
	\frac{
		\sin(c_{\mathrm D}|\xi|^2t)
	}{
		c_{\mathrm D}|\xi|^2
	}
	\right)
	\right\|_{L^2_{\chi}}^2
	\left|
	M_{\varphi_1}
	-
	\frac{I_{\rho}}{\rho}
	P_{\psi_1}
	\right|^2
	-
	o(t^{\frac{1}{2}})
	\\
	&\gtrsim
	t^{\frac{1}{2}}
	\left|
	M_{\varphi_1}
	-
	\frac{I_{\rho}}{\rho}
	P_{\psi_1}
	\right|^2
\end{align*}
as
$t\to+\infty$.
Therefore,
we complete the proof
of the lower bound estimate
under the condition
$
M_{\varphi_1}
-
\frac{I_{\rho}}{\rho}
P_{\psi_1}
\neq0.
$
Moreover,
the above derivation also yields
the corresponding large time asymptotic profile.
Since its proof follows exactly the same argument as in
\eqref{Est-02},
we omit the details.

\subsection{Proof of Corollary \ref{Coro-L-infty}}

\hspace{5mm}
As a consequence of the hyperbolic structure
of the classical Timoshenko system
\eqref{Eq-Timoshenko},
one may prove the finite propagation property
by the standard localized energy method,
analogously to the classical wave equation.

\begin{prop}\label{Prop-FPPS}
	Suppose that
	\begin{align*}
		\mathrm{supp}\,
		(\varphi_0,\varphi_1,\psi_0,\psi_1)
		\subset
		[-L,L]
		\ \ \mbox{for some}\ \ 
		L>0
	\end{align*}
	for the classical Timoshenko system
	\eqref{Eq-Timoshenko}.
	Then,
	the transversal displacement
	$\varphi$
	satisfies
	\begin{align*}
		\mathrm{supp}\,
		\varphi(t,\cdot)
		\subset
		[-L-c_*t,L+c_*t]
	\end{align*}
	for every
	$t\geqslant0$,
	where 
	$
	c_*
	:=
	\max\{
	c_{\mathrm S},
	c_{\mathrm R}
	\}$.
\end{prop}

\noindent
According to Proposition \ref{Prop-FPPS},
the support of
$\varphi(t,\cdot)$
is contained in an interval
of length
$2(L+c_*t)$.
Therefore,
for large time
$t\gg1$, H\"older's inequality with $\frac{1}{p}+\frac{1}{q}=1$ and $q\geqslant 1$ implies
\begin{align*}
	\|\varphi(t,\cdot)\|_{L^2}^2
	=
	\int_{-L-c_*t}^{L+c_*t}
	|\varphi(t,x)|^2
	\,\mathrm dx
&\lesssim
	(L+c_*t)^{\frac{1}{p}}
	\|\varphi(t,\cdot)\|_{L^{2q}}^{2}
	\\
	&\lesssim
	t^{\frac{1}{p}}
	\|\varphi(t,\cdot)\|_{L^{2q}}^2.
\end{align*}
Consequently, the lower bound estimate in Theorem \ref{Thm-02} indicates
\[
\|\varphi(t,\cdot)\|_{L^{2q}}
\gtrsim
t^{\frac{1}{2q}-\frac{1}{2}}
\|\varphi(t,\cdot)\|_{L^2}
\gtrsim
t^{\frac{1}{4}+\frac{1}{2q}}
|P_{\varphi_1}|,
\]
which completes the proof.

\section{Further remarks: Relation to the dissipative Timoshenko system}\setcounter{equation}{0}\label{Sec-Final}

\hspace{5mm}
In Section \ref{Section-Introduction}, we pointed out that the singular low-frequency structure of the classical Timoshenko system generates an intrinsic plate-type asymptotic profile and is responsible for the polynomial growth structure of solutions.
Moreover, as discussed in Remark \ref{Rem-2.8}, the frictional damping mechanism in the dissipative Timoshenko system \eqref{Eq-dissipative-Timoshenko} modifies the corresponding profile through a diffusive effect, while preserving the underlying plate-type structure. A similar phenomenon for the free wave equation and the strongly damped wave equation
has recently been observed in \cite{Ikehata-Takeda=2026}.

Motivated by these observations,
we next investigate
the relation between
the classical Timoshenko system
\eqref{Eq-Timoshenko}
and the dissipative Timoshenko system
\eqref{Eq-dissipative-Timoshenko}
in the small damping regime
\begin{align}\label{Damping-Regime}
0<\gamma<2\sqrt{\rho EI},
\end{align}
where the low-frequency oscillatory structure
is preserved. Throughout this section,
$(\varphi^{\gamma=0},\psi^{\gamma=0})$ and $(\varphi^{\gamma>0},\psi^{\gamma>0})$
denote the solutions to \eqref{Eq-Timoshenko} and \eqref{Eq-dissipative-Timoshenko},
respectively, associated with the same initial data $(\varphi_0,\varphi_1)$ as well as $(\psi_0,\psi_1)$,
for which no superscripts are used.

More precisely,
we establish
a large time vanishing dissipation limit
for time-normalized solutions.
After normalization
by the natural growth rates,
the solutions to
the dissipative Timoshenko system \eqref{Eq-dissipative-Timoshenko}
converge to those of the classical Timoshenko system \eqref{Eq-Timoshenko} as $\gamma\to0$ in the large time regime.

\subsection{Large time behavior for the dissipative Timoshenko system}

\hspace{5mm}
By applying the same low-frequency analysis
as in \cite{Chen=2026}
to the dissipative Timoshenko system
\eqref{Eq-dissipative-Timoshenko},
one obtains the following large time asymptotic behavior
in the small damping regime \eqref{Damping-Regime}, which coincides with
\cite[Theorem 2.1 and Theorem 2.2]{Chen=2026}.
We introduce the notations
\begin{align*}
	c_{\gamma}
	:=
	\frac{\sqrt{4\rho EI-\gamma^2}}{2\rho}
	\ \ \mbox{and}\ \ 
	\delta_{\gamma}
	:=
	\frac{\gamma}{2\rho}.
\end{align*}

\begin{prop}\label{Prop-01}
	Suppose that
	$(\varphi_0,\varphi_1)\in Z_1^{0}$
	and
	$(\psi_0,\psi_1)\in Z_{2}^{0}$
	such that
	$P_{\varphi_1}\neq0$
	for the dissipative Timoshenko system
	\eqref{Eq-dissipative-Timoshenko}
	with
	$0<\gamma<2\sqrt{\rho EI}$.
	Then, 	the transversal displacement
	$\varphi^{\gamma>0}$ satisfies the following optimal growth estimate:
	\begin{align*}
		t^{\frac34}|P_{\varphi_1}|
		\lesssim
		\|\varphi^{\gamma>0}(t,\cdot)\|_{L^2}
		\lesssim
		t^{\frac34}
		\|(\varphi_0,\varphi_1)\|_{Z_1^0}
		+
		t^{\frac14}
		\|(\psi_0,\psi_1)\|_{Z_2^0}
	\end{align*}
	for sufficiently large time,
	and the following asymptotic relation:
	\begin{align*}
		\lim\limits_{t\to+\infty}
		t^{-\frac34}
		\left\|
		\varphi^{\gamma>0}(t,\cdot)
		-
		\sqrt t\,
		\widetilde{\ml{G}}_{0}
		\left(
		\tfrac{\cdot}{\sqrt t};\gamma
		\right)
		P_{\varphi_1}
		\right\|_{L^2}
		=
		0,
	\end{align*}
	where the diffusion plate-type asymptotic profile is given by
	\begin{align*}
		\widetilde{\ml{G}}_{0}(y;\gamma)
		:=
		\ml{F}^{-1}_{\eta\to y}
		\left(
		\frac{
			\sin(c_{\gamma}|\eta|^2)
		}{
			c_{\gamma}|\eta|^2
		}
		\,
		\mathrm e^{-\delta_{\gamma}|\eta|^2}
		\right).
	\end{align*}
\end{prop}

\begin{prop}\label{Prop-02}
	Suppose that
	$(\varphi_0,\varphi_1)\in Z_2^{0}$
	and
	$(\psi_0,\psi_1)\in Z_{3}$
	such that
	$P_{\varphi_1}\neq0$
	for the dissipative Timoshenko system
	\eqref{Eq-dissipative-Timoshenko}
	with
	$0<\gamma<2\sqrt{\rho EI}$.
	Then,
	the rotation angle
	$\psi^{\gamma>0}$
	satisfies the following optimal growth estimate:
	\begin{align*}
		t^{\frac14}|P_{\varphi_1}|
		\lesssim
		\|\psi^{\gamma>0}(t,\cdot)\|_{L^2}
		\lesssim
		t^{\frac14}
		\|(\varphi_0,\varphi_1)\|_{Z_2^0}
		+
		\|(\psi_0,\psi_1)\|_{Z_3}
	\end{align*}
	for sufficiently large time,
	and the following asymptotic relation:
	\begin{align*}
		\lim\limits_{t\to+\infty}
		t^{-\frac14}
		\left\|
		\psi^{\gamma>0}(t,\cdot)
		-
		\widetilde{\ml{G}}_{1}
		\left(
		\tfrac{\cdot}{\sqrt t};\gamma
		\right)
		P_{\varphi_1}
		\right\|_{L^2}
		=
		0,
	\end{align*}
	where the derivative of diffusion plate-type asymptotic profile is given by
	\begin{align*}
		\widetilde{\ml{G}}_{1}(y;\gamma)
		:=
		\partial_y
		\widetilde{\ml{G}}_{0}(y;\gamma)
		=
		\ml{F}^{-1}_{\eta\to y}
		\left(
		i\eta
		\frac{
			\sin(c_{\gamma}|\eta|^2)
		}{
			c_{\gamma}|\eta|^2
		}
		\,
		\mathrm e^{-\delta_{\gamma}|\eta|^2}
		\right).
	\end{align*}
\end{prop}

\begin{remark}
	The asymptotic profiles
	$\widetilde{\ml{G}}_0(y;\gamma)$
	and
	$\widetilde{\ml{G}}_1(y;\gamma)$
	consist of two different components:
	\begin{itemize}
		\item the oscillatory plate-type phase $	\frac{
			\sin(c_\gamma |\eta|^2)}{c_\gamma |\eta|^2}$,
			\item the diffusive factor $\mathrm e^{-\delta_\gamma |\eta|^2}$.
	\end{itemize}
	The oscillatory structure is inherited from the conservative Timoshenko system \eqref{Eq-Timoshenko},
	whereas the Gaussian factor is generated by the frictional damping mechanism.
	In particular,
	as $\gamma\to0$,
	one has
	\[
	c_\gamma\to c_{\mathrm D}\ \ \mbox{and}\ \ \delta_\gamma\to0,
	\]
	so that the dissipative profiles converge formally to the conservative profiles
	$\ml{G}_0(y)$ and $\ml{G}_1(y)$
	derived in Theorem \ref{Thm-01} and Theorem \ref{Thm-02}.
\end{remark}

\begin{remark} The restriction $ 0<\gamma<2\sqrt{\rho EI}$ corresponds to the regime where the low-frequency characteristic roots of the dissipative Timoshenko system \eqref{Eq-dissipative-Timoshenko} remain complex-valued. Consequently, the oscillatory plate-type structure persists under the frictional damping mechanism, which makes it possible to compare the dissipative dynamics with those of the classical Timoshenko system \eqref{Eq-Timoshenko} in the vanishing dissipation limit. \end{remark}

\subsection{Large time vanishing dissipation limit}

\hspace{5mm}
The next result shows that the frictional damping modifies the leading asymptotic behavior only through a diffusive factor and a shifted oscillation frequency, while preserving the singular plate-type structure responsible for the large time growth of solutions. The following statement is formulated under a common set of assumptions guaranteeing that all asymptotic profiles appearing in Theorem \ref{Thm-01}, Theorem \ref{Thm-02}, Proposition \ref{Prop-01}, and Proposition \ref{Prop-02} are well-defined simultaneously.

\begin{theorem}\label{Thm-3}
	Suppose that $(\varphi_0,\varphi_1)\in Z_1^{0}$ and $(\psi_0,\psi_1)\in Z_{1}^0$ such that $P_{\varphi_1}\neq0$ for the Timoshenko systems \eqref{Eq-Timoshenko} and \eqref{Eq-dissipative-Timoshenko}. Then,
	the transversal displacements and the rotation angles, respectively, satisfy the following asymptotic relations:
	\begin{align*}
		\lim\limits_{\gamma\to 0}\lim\limits_{t\to+\infty}t^{-\frac{3}{4}}\left\|\varphi^{\gamma>0}(t,\cdot)-\varphi^{\gamma=0}(t,\cdot)\right\|_{L^2}=0,\\[0.5em]
		\lim\limits_{\gamma\to 0}\lim\limits_{t\to+\infty}t^{-\frac{1}{4}}\left\|\psi^{\gamma>0}(t,\cdot)-\psi^{\gamma=0}(t,\cdot)\right\|_{L^2}
		=0.
	\end{align*}
\end{theorem}

\begin{proof}
	Let us begin with the first asymptotic relation for $\varphi^{\gamma>0}$ and $\varphi^{\gamma=0}$.
	By using the triangle inequality associated with Proposition \ref{Prop-01} and Theorem \ref{Thm-01}, one notices
	\begin{align*}
		t^{-\frac{3}{4}}\left\|\varphi^{\gamma>0}(t,\cdot)-\varphi^{\gamma=0}(t,\cdot)\right\|_{L^2}
		&\leqslant t^{-\frac{3}{4}}\left\|\varphi^{\gamma>0}(t,\cdot)-\sqrt{t}\,\widetilde{\ml{G}}_{0}\left(\tfrac{\cdot}{\sqrt t};\gamma\right)P_{\varphi_1}\right\|_{L^2}\\
		&\quad+t^{-\frac{1}{4}}\left\|\widetilde{\ml{G}}_{0}\left(\tfrac{\cdot}{\sqrt t};\gamma\right)-\ml{G}_{0}\left(\tfrac{\cdot}{\sqrt t}\right)\right\|_{L^2}|P_{\varphi_1}|\\
		&\quad+t^{-\frac{3}{4}}\left\|\varphi^{\gamma=0}(t,\cdot)-\sqrt{t}\,\ml{G}_{0}\left(\tfrac{\cdot}{\sqrt t}\right)P_{\varphi_1}\right\|_{L^2}\\
		&\leqslant
		\underbrace{\left\|
		\frac{\sin(c_{\gamma}|\eta|^2)}{c_{\gamma}|\eta|^2}
		\,\mathrm{e}^{-\delta_{\gamma}|\eta|^2}-\frac{\sin(c_{\mathrm D}|\eta|^2)}{c_{\mathrm D}|\eta|^2}
		\right\|_{L^2}}_{=:A_1(\gamma)}|P_{\varphi_1}|+o(1)
	\end{align*}
	for large time
	$t\gg1$,
	where we used the scaling property of the $L^2$ norm with $\eta=\sqrt{t}\,\xi$.
	
	Since
	$c_{\gamma}\to c_{\mathrm D}$
	and
	$\delta_{\gamma}\to0$
	as
	$\gamma\to0$,
	the integrand converges pointwise to zero.
	Moreover,
	\begin{align*}
		[A_1(\gamma)]^2
		&=\int_{\mathbb R}
		\left|\frac{\sin(c_{\gamma}|\eta|^2)
		}{
			c_{\gamma}|\eta|^2
		}
		\,\mathrm{e}^{-\delta_{\gamma}|\eta|^2}
		-
		\frac{
			\sin(c_{\mathrm D}|\eta|^2)
		}{
			c_{\mathrm D}|\eta|^2
		}
		\right|^2
		\mathrm d\eta
		\\
		&\lesssim
		\int_{|\eta|\leqslant1}
		\mathrm d\eta
		+
		\int_{|\eta|\geqslant1}
		|\eta|^{-4}
		\,\mathrm d\eta
		<+\infty.
	\end{align*}
	Therefore, the dominated convergence theorem yields $\lim\limits_{\gamma\to0}A_1(\gamma)=0$.
	Consequently,
	\begin{align*}
		\lim\limits_{\gamma\to0}
		\lim\limits_{t\to+\infty}
		t^{-\frac{3}{4}}
		\left\|
		\varphi^{\gamma>0}(t,\cdot)
		-
		\varphi^{\gamma=0}(t,\cdot)
		\right\|_{L^2}
		=0.
	\end{align*}
	
	Concerning the second asymptotic relation,
	we follow the same argument associated with Proposition \ref{Prop-02} and Theorem \ref{Thm-02} to derive our aim.
	This completes the proof.
\end{proof}

\begin{remark}
	Theorem \ref{Thm-3} shows that,
	after normalization by the intrinsic growth rates
	$t^{\frac34}$
	and
	$t^{\frac14}$,
	the dissipative Timoshenko system \eqref{Eq-dissipative-Timoshenko} converges to the classical Timoshenko system \eqref{Eq-Timoshenko} in the small damping regime \eqref{Damping-Regime}.
	Therefore,
	the frictional damping modifies only the profile shape through the diffusive factor
	$\mathrm{e}^{-\delta_{\gamma}|\eta|^2}$
	and the shifted oscillation frequency
	$c_{\gamma}$,
	while the singular plate-type growth mechanism itself remains stable as
	$\gamma\to0$.
\end{remark}

\begin{remark}
	The large time limit and the vanishing dissipation limit
	are taken successively in Theorem \ref{Thm-3}.
	Since the dissipative correction
	appears through the factor
	$\mathrm e^{-\delta_\gamma |\eta|^2}$,
	the corresponding asymptotic regimes
	depend on the interaction
	between the parameters
	$\gamma$ and $t$.
	The analysis of simultaneous limits,
	for example in the transitional regime
	$\gamma \,t\approx 1$,
	remains an interesting open problem.
\end{remark}

\section*{Acknowledgments}
Wenhui Chen is supported in part by the National Natural Science Foundation of China (grant No. 12301270), Guangdong Basic and Applied Basic Research Foundation (grant No. 2025A1515010240).  The author thanks Ryo Ikehata (Hiroshima University) for some suggestions in the preparation of this manuscript.

\end{document}